\documentclass[11pt]{article}
 \usepackage{bm,mathtools}
 \usepackage[a4paper, margin=2cm]{geometry}
  \usepackage{enumitem}
  \usepackage[hidelinks]{hyperref}
 \newtheorem{theorem}{Theorem}[section]
 \newtheorem{lemma}[theorem]{Lemma}
 \newtheorem{proposition}[theorem]{Proposition}
 \newtheorem{corollary}[theorem]{Corollary}

 \newtheorem{remark}[theorem]{Remark}
 \newenvironment{proof}{{\it Proof}. }{\hfill\END\\[0.5ex]}

\DeclareMathOperator{\vcurl}{\overrightarrow{\mathrm{curl}}}
\DeclareMathOperator{\curl}{curl}
\DeclareMathOperator{\divv}{div}

\DeclareMathOperator{\SL}{SL}
\DeclareMathOperator{\DL}{DL}

\newcommand{\D}{\mathrm{D}}
\newcommand{\J}{\mathrm{J}}
\DeclareMathOperator{\V}{V}
\DeclareMathOperator{\W}{W}
\DeclareMathOperator{\K}{K}
\DeclareMathOperator{\HH}{H}
\DeclareMathOperator{\I}{I}

\DeclareMathOperator{\sign}{sign}

\usepackage{tikz}

\usepackage{multirow}

\usepackage[final]{changes}
\definechangesauthor[name=vD, color=red]{vD}
\definechangesauthor[name=cat, color=blue]{cat}

\usepackage{pgfplots}
\usepackage{makecell}

\usepackage{amsmath}
\usepackage{amssymb}

\usepackage[normalem]{ulem}

\newcommand{\R}{\mathbb{R}}

\newcommand{\be}{\begin{equation}}
\newcommand{\ee}{\end{equation}}

\newcommand{\x}{{\bm{x}}}
\newcommand{\y}{{\bm{y}}}

\newcommand{\nnn}{\bm{n}}
\newcommand{\ttt}{\bm{t}}

 \makeatletter\@addtoreset{equation}{section}\makeatother

\newcommand{\END}{\hfill$\Box$}

\usepackage{mathabx}

\title{Stability estimates of Nystr\"om discretizations of Helmholtz decomposition boundary integral equation formulations for the solution of Navier scattering problems in two dimensions with Dirichlet boundary conditions}
\date{\today}
\author{V\'{\i}ctor Dom\'{\i}nguez\thanks{Dep. Estadística, Informática y Matemáticas. Campus de Tudela 31500 - Tudela, Spain. INAMAT2 Institute for Advanced Materials and Mathematics. e-mail: victor.dominguez@unavarra.es. }\and Catalin Turc\thanks{  Department of
Mathematical Sciences, New Jersey  Institute of Technology,
Univ. Heights. 323 Dr. M. L. King Jr. Blvd, Newark, NJ 07102, USA, e-mail: catalin.c.turc@njit.edu.}}
\allowdisplaybreaks

\begin{document}
\maketitle
\begin{abstract}
 Helmholtz decompositions of elastic fields is a common approach for the solution of Navier scattering problems. Used in the context of Boundary Integral Equations (BIE), this approach affords solutions of Navier problems via the simpler Helmholtz boundary integral operators (BIOs). Approximations of Helmholtz Dirichlet-to-Neumann (DtN) can be employed within a regularizing combined field strategy to deliver BIE formulations of the second kind for the solution of Navier scattering problems in two dimensions with Dirichlet boundary conditions, at least in the case of smooth boundaries. Unlike the case of scattering and transmission Helmholtz problems, the approximations of the DtN maps we use in the Helmholtz decomposition BIE in the Navier case require incorporation of lower order terms in their pseudodifferential asymptotic expansions. The presence of these lower order terms in the Navier regularized BIE formulations complicates the stability analysis of their Nystr\"om discretizations in the framework of global trigonometric interpolation and the Kussmaul-Martensen kernel singularity splitting strategy. The main difficulty stems from compositions of pseudodifferential operators of opposite orders, whose Nystr\"om discretization must be performed with care via pseudodifferential expansions beyond the principal symbol. The error analysis is significantly simpler in the case of arclength boundary parametrizations and considerably more involved in the case of general smooth parametrizations which are typically encountered in the description of one dimensional closed curves.
  \newline \indent
  \textbf{Keywords}: Time-harmonic Navier scattering problems, Helmholtz decomposition, boundary integral equations, pseudodifferential calculus, Nystr\"om discretizations, regularizing operators.\\

 \textbf{AMS subject classifications}:
 65N38, 35J05, 65T40, 65F08
\end{abstract}

\section{Introduction}

Elastic waves in homogeneous media, that is solutions of the time harmonic Navier equations, can be expressed via Helmholtz decompositions as linear superpositions of P-waves and S-waves, which are in turn solutions of Helmholtz equations with different wavenumbers related to the Lam\'e constants of the medium in which they propagate. The enforcement of Dirichlet boundary conditions in this approach leads to coupled boundary conditions for the pressure and the shear waves featuring normal and tangential derivatives on the boundary of the scatterer. This simple observation affords the possibility to use Helmholtz potentials corresponding to the pressure and shear wave numbers and their related boundary integral operators (BIOs) to solve time harmonic Navier scattering problems in homogeneous media. This approach has to be contrasted to the classical BIE approaches based on Navier Green's function potentials~\cite{dominguez2008dirac,dominguez2012fully,dominguez2014nystrom,DoSaSa:2015,faria2021general,chapko2000numerical,dominguez2021boundary}.

The combined field approach~\cite{BrackhageWerner}, which is the method of choice to deliver robust BIE formulations for Helmholtz scattering problems at any frequencies, leads in this context to systems of BIE of the first kind (pseudodifferential operators of order one) owing to the presence of Helmholtz hypersingular BIOs. To compound the difficulty, the principal symbols (in the pseudodifferential operator sense) of these combined field formulations are defective matrix operators~\cite{dong2021highly,DomTurc:2023}. The regularized combined field methodology~\cite{boubendir2015regularized} relies on approximations of Dirichlet to Neumann ({\rm DtN}) operators to construct intrinsically well conditioned boundary integral equations, which, in the case of smooth boundaries, are always of the second kind. This strategy was successfully applied to the design of well conditioned Navier BIE by incorporating suitable approximations of Navier DtN operators~\cite{chaillat2015approximate,chaillat2020analytical,dominguez2021boundary,osti_10185828}.  The Helmholtz decomposition approach bypasses the need to construct the latter, as only Helmholtz DtN are required~\cite{DomTurc:2023}. However, in order to mitigate the aforementioned defective property, the DtN approximations require incorporating lower order terms in the pseudodifferential asymptotic expansion of DtN operators~\cite{DomTurc:2023}. Leveraging arclength parametrization, the pseusodifferential calculus beyond the principal symbol can be carried out within the framework of periodic Fourier multipliers. This calculus gets more delicate in the case of arbitrary parametrizations, and we present in this paper such a generalization that is invariant of the parametrization.

We introduce and analyze stable Nystr\"om discretizations of the regularized BIE for the solutions of Navier scattering problems in the Helmholtz decomposition framework. The stable Nystr\"om discretization of the composition of pseudodifferential operators of opposite orders that typically feature in regularized BIE formulations must be done with care~\cite{BoTuDo:2015}. A stable approach for handling such compositions within the global trigonometric framework involves isolating the principal parts of pseudodifferential operators with opposite orders and explicitly composing them using Fourier multiplication calculus. Once the principal parts were dealt with, the composition of pseudodifferential operators of negative orders is benign, and can be realized by simple multiplication of Nystr\"om discretization matrices. Our Nystr\"om approach relies on Kussmaul-Martensen kernel logarithmic splittings  \cite{Kussmaul,MR0147661}, which leads naturally to constructing asymptotic pseudodifferential expansions of the Helmholtz BIOs and therefore to carrying out operator compositions per the prescriptions above. { This approach was adopted and developed in the 80s of the previous century  by R. Kress and collaborators to provide a full discretization of the four layer operators for Helmholtz equations for which a rigorous proof of stability and convergence was given in the Sobolev framework cf  \cite{Kress,KressH} (see also \cite{MR3463450,DomCat:2015}). For Helmholtz equation, the resulting scheme is known to be superalgebraically convergent, that is, faster than any negative power of the number of points used in this discretization, provided that the exact solution is smooth. We stress that same convergence order is stated in this work for this more complex equation  {(see Theorems \ref{theo:4.5} and \ref{theo:5.4})}.

The error analysis of the Nystr\"om discretizations of the regularized BIE considered in this paper is conducted along the guidelines above, with extra complications that are inevitable to the presence of several terms in the pseudodifferential expansions of some of the key BIOs in our formulations, including the regularizing operators and the DtN approximations themselves. We chose to present the analysis both in the simpler arclength parametrization framework, as well as in the more complicated, yet more realistic case of general smooth parametrization of closed boundary curves. The use of smooth but non-arc length parametrizations is often necessary. In practice, it is not difficult to compute numerically sufficiently accurate approximations of the arc parametrization, even for rather complex geometries. However, such parametrizations result in an approximately uniform distribution of points along the curve, failing to distinguish between regions of rapid or slow variation in the curve (i.e., regions of high or low curvature). Since some of the most popular numerical schemes, such as the superalgebraically convergent Kress-Nystr\"om discretization of the Helmholtz boundary integral operators we consider in this paper, use second and even third derivatives of the parametrization in the definition, this results in poor performance and/or very high sensitivity to round-off errors arising in the splitting of the kernels of the integral operators into smooth and singular parts. We provide evidence of these phenomena in the numerical experiments section.

We also believe that the design of other discretizations -such as those based on Petrov-Galerkin, Collocation or Qualocation spline based methods (see \cite{MR0717691,MR1626332,MR0958484,Saranen} and references therein) or even deltaBEM schemes \cite{MR3141718,deltaBEM}- and their mathematical analysis, are also possible up to some extent. However, in addition to the high order the Nystr\"om method enjoys,  there is a notable feature that makes the method considered and fully analysed here attractive: Both the regularizer and the approximation we propose for the Dirichlet-to-Neumann operator used in the  pressure and shear wave  representation formula are  diagonal Fourier operators (i.e., convolution operators) for arc-length parametrizations of the curve  which favours a simple integration with the solver. For general parametrizations this is only true for their principal parts, which is still advantageous. On the other hand,  we are confident that Nystr\"om discretization  can also be analysed for the Neumann problem -the integral equation is in this case of order 1 (i.e. it behaves as a derivative)- by employing similar tools as those we have developed in this paper in a more complex and/or technical analysis. Numerical experiments support clearly this conjecture (see \cite{DomTurc:2023}).

The paper is organized as follows: in Section~\ref{setting} we introduce the boundary value problem for Navier equations in two dimensions and we present the Helmholtz decomposition approach and a first boundary integral formulation for this problem.  In Section~\ref{Helm} we rewrite the boundary integral equations as periodic operator equations by means of arc-length parametrizations of the boundary of the Navier problem domain. The required regularizer  is then introduced and the well-posedness of the equation is stated. In Section~\ref{err_anal} we introduce  the  Nystr\"om method and prove stability and superalgebraically convergence for this equation; in Section~\ref{gen} we construct the regularizing operators for general parametrizations, we extend the Nystr\"om discretization for the regularized formulations and show that stability and superalgebraically convergence is preserved.
Finally, we present in Section~\ref{num} a few numerical experiments to illustrate the second kind nature of the regularized formulations, the superalgebraically convergence proved in this paper as well as comparisons between the iterative behavior of Nystr\"om solvers for both types of parametrizations considered in this paper.

\section{Navier equations and boundary element method}\label{setting}

\subsection{Boundary value problem and Helmholtz decomposition}

For any vector function ${\bf u}=(u_1,u_2)^\top:\mathbb{R}^2\to   \mathbb{R}^2$ (vectors in this paper will be always regarded as {\em column} vectors)  the strain tensor in  a linear isotropic and homogeneous elastic medium with Lam\'e constants $\lambda$ and $\mu$  is defined as
\[
 \bm{\epsilon}({\bf u}):=\frac{1}2 (\nabla {\bf u}+(\nabla {\bf u})^\top)
 =\begin{bmatrix}
  \partial_{x_1} u_1& \tfrac{1}2\left(\partial_{x_1} u_2+\partial_{x_2} u_1\right)     \\
   \tfrac{1}2\left(\partial_{x_1} u_2+\partial_{x_2} u_1\right) &
   \partial_{x_2} u_2
     \end{bmatrix}.
\]
The stress  tensor is then given by
\[
 \bm{\sigma}({\bf u}):=2\mu \bm{\epsilon}({\bf u})+\lambda (\divv{\bf u}) I_2
\]
where $I_2$ is the identity matrix of order 2 and the Lam\'e coefficients  $\lambda$  are assumed to satisfy $\lambda,\ \lambda +2\mu>0$.  For $\Omega$ a smooth bounded domain with boundary $\Gamma$, the exterior Dirichlet problem for the time-harmonic elastic wave (Navier) equation is
\begin{equation}\label{eq:NavD}
\left|
\begin{array}{l}
 \divv\bm{\sigma}({\bf u})+\omega^2{\bf u}
 =0,\quad \text{in }\Omega^+:=\mathbb{R}^2\setminus{\Omega},\\
 \gamma_{\Gamma} {\bf u} = -{\bf u}^{\rm inc},\\
 \text{+Kupradze radiation condition}
 \end{array}
 \right.
\end{equation}
where the frequency $\omega\in\mathbb{R}^+$ and the {\rm divergence} operator ${\rm div}$ is applied row-wise. Here ${\bf u}^{\rm inc}$ is a solution of the Navier equation in a neighborhood of  $\overline{\Omega}$, typically in $\mathbb{R}^2$, although point source elastic waves are also supported.

The Kupradze radiation condition at infinity \cite{AmKaLe:2009,kupradze2012three} can be described as follows: if
\begin{equation}\label{eq:Hdecomp1}
\mathbf{u}_p:=-\frac{1}{k_p^2} \nabla \divv \mathbf{u}, \quad \mathbf{u}_s:=\mathbf{u}-\mathbf{u}_p= \vcurl {\rm{curl} \ {\bf u}}
\end{equation}
($\vcurl \varphi :=(\partial_{x_2} \varphi,-\partial_{x_1}\varphi)$, $\curl{\bf u} := \partial_{x_2}u_1-\partial_{x_1} u_2$ are respectively the vector and scalar curl, or rotational, operator)
with
\begin{equation}\label{eq:ks:kp}
k_p^2:=\frac{\omega^2}{\lambda+2 \mu}, \quad k_s^2:=\frac{\omega^2}{\mu}
\end{equation}
the associated the pressure and stress wave-numbers wave-numbers, then
\[
\frac{\partial \mathbf{u}_p}{\partial \widehat{\bm{x}}}(\bm{x})-i k_p \mathbf{u}_p(\bm{x}),\quad \frac{\partial \mathbf{u}_s}{\partial \widehat{\bm{x}}}(\bm{x})-i k_s \mathbf{u}_s(\bm{x})=o\left(|\bm{x}|^{-1 / 2}\right), \quad
 \widehat{\bm{x}} := \frac{1}{|\bm{x}|}\bm{x}.
\]

 In view of \eqref{eq:Hdecomp1}, a common approach to the solution of Navier equations is to look for the fields ${\bf u}$ in the form
\begin{equation}\label{eq:Hdecomp2}
{\bf u}=\nabla u_p+\vcurl  {u_s}
\end{equation}
where     $u_p$ and $u_s$ are respectively solutions of the Helmholtz equations in $\Omega^+$  with wave-numbers $k_p$ and $k_s$  satisfying the Sommerfeld radiation condition at infinity which is in itself a simple consequence of the Kupradze radiation condition. Specifically, the scalar fields in the Helmholtz decomposition are solutions of the following Helmholtz problems
\begin{equation}\label{eq:2.5}
 \left|
 \begin{array}{ll}
  \Delta u_p +k_p^2 = 0, \quad \text{in }\Omega^+\\
\partial_{\widehat{\bm{x}}}u_p-ik_p u_p =o(|\bm{x}|^{-1/2}), \quad \text{as }
|{\bm x}|\to\infty
 \end{array}
 \right.\quad
 \left|
 \begin{array}{ll}
  \Delta u_s +k_s^2 = 0, \quad \text{in }\Omega^+\\
\partial_{\widehat{\bm{x}}}u_s-ik_s u_s =o(|\bm{x}|^{-1/2}), \quad \text{as }
|{\bm x}|\to\infty.
 \end{array}\right.
\end{equation}
If $\nnn$ denotes the outward unit normal vector to $\Gamma$, $\ttt$, the positively (counterclockwise) oriented tangent field given by
\begin{equation}\label{eq:2.6}
 \ttt =-{\rm Q}\nnn, \quad  {\rm Q}:=\begin{bmatrix}
                                    0 & 1\\
                                   -1 &0
                                  \end{bmatrix}
\end{equation}
which satisfy
\[
\begin{aligned}
\nabla{\bf u}_p\cdot\nnn  &= \partial_{\nnn} u_p  \quad &
\nabla{\bf u}_p\cdot\ttt  &= \partial_{\ttt} u_p  \\
\vcurl{\bf u}_s\cdot\nnn  &= \partial_{\ttt} u_s  \quad &
\vcurl{\bf u}_s\cdot\ttt  &= -\partial_{\nnn} u_s
\end{aligned}
\]
($\partial_{\nnn}$ and $\partial_{\ttt}$ are then the exterior normal and the positively oriented tangent derivative), the Dirichlet condition in problem \eqref{eq:NavD} leads to the following boundary conditions on $\Gamma$ for $u_p$ and $u_s$:
\begin{equation}\label{eq:2.7}
\begin{array}{rcll}
\partial_{\nnn}u_p+\partial_{\ttt}u_s&=&-{\bf u}^{\rm inc}\cdot \nnn,& {\rm on}\ \Gamma, \\
\partial_{\ttt}u_p-\partial_{\nnn}u_s&=&-{\bf u}^{\rm inc}\cdot \ttt,& {\rm on}\ \Gamma.
\end{array}
\end{equation}
The reformulation of the Navier scattering problem with Dirichlet boundary conditions in the Helmholtz decomposition framework is readily amenable to boundary integral formulations, as we will explain in what follows.


%
 \begin{remark}\label{remark:curves:2pi} Throughout this article we will assume that $\Gamma$, the boundary of the domain $\Omega^+$, is of length $2\pi$. The general case can be reduced to this particular scenario  by replacing the wawenumber(s)   $k$, and the complexifications $\widetilde{k}$ we will introduce later, by $ Lk/(2\pi)$, its characteristic length, and $L\widetilde{k}/(2\pi)$, where $L$ is the length of the curve.
 \end{remark}

\subsection{Helmholtz BIOs}\label{GradHessCalculations}

For a given wave-number $k>0$ and a functional density $\varphi_\Gamma$ on the boundary $\Gamma$ we define the Helmholtz single and double layer potentials in the form
\[
 \SL_{k,\Gamma}[\varphi_\Gamma](\x):=\int_\Gamma \phi_k(\x-\y)\varphi_\Gamma(\y) {{\rm d}\y},\quad
 \DL_{k,\Gamma}[\varphi_\Gamma](\x):=\int_\Gamma \frac{\partial \phi_k(\x-\y)}{\partial \nnn(\y)}\varphi_\Gamma(\y) {{\rm d}\y},\ \x\in\mathbb{R}^2\setminus\Gamma
\]
with
\[
 \phi_k(\x) = \frac{i}4 H_0^{(1)}(k|\x|)
\]
 the fundamental solution of the Helmholtz equation  ($H_0^{(1)}$ is then the Hankel function of first kind and order zero).

The four BIOs of the Calder\'on's calculus associated with the Helmholtz equation are defined by applying the exterior/interior Dirichlet and Neumann traces on $\Gamma$ (denoted in what follows by $\gamma_{\Gamma}^+/\gamma_{\Gamma}^-$ and $\partial_{\nnn}^+/\partial_{\nnn}^-$ respectively) to the Helmholtz single and double layer potentials cf. \cite{hsiao2008boundary,mclean:2000,Saranen}. Specifically,
\begin{equation}\label{eq:traces}
\begin{aligned}
\gamma_{\Gamma}^\pm  \SL_{k,\Gamma}\varphi_\Gamma &=\V_{k,\Gamma}\varphi_\Gamma, \quad & \partial_{\nnn}^\pm  \SL_{k,\Gamma}\varphi_\Gamma &=\mp\frac{1}{2}\varphi_\Gamma +\K_{k,\Gamma}^\top\varphi_\Gamma,  \\
\gamma_{\Gamma}^\pm \DL_{k,\Gamma}\varphi_\Gamma &=\pm\frac{1}{2}\varphi_\Gamma+\K_{k,\Gamma}\varphi_\Gamma, \quad & \partial_{\nnn}^\pm \DL_{k,\Gamma}\varphi_\Gamma &=\W_{k,\Gamma}\varphi_\Gamma
\end{aligned}
\end{equation}
where, for $\x\in\Gamma$,
\begin{eqnarray}
( \V_{k,\Gamma}\varphi_\Gamma)(\x) &:=&\int_\Gamma \phi_k(\x-\y)\varphi_\Gamma(\y) {{\rm d}\y},\label{eq:V}\\
( \K_{k,\Gamma}\varphi_\Gamma)(\x) &:=&\int_\Gamma \frac{\partial \phi_k(\x-\y)}{\partial \nnn(\y)}\varphi_\Gamma(\y) {{\rm d}\y},\label{eq:Kt}\\
( \K_{k,\Gamma}^\top\varphi_\Gamma)(\x)  &:=& \int_\Gamma \frac{\partial \phi_k(\x-\y)}{\partial \nnn(\x)}\varphi_\Gamma(\y) {{\rm d}\y},\label{eq:K}\\
( \W_{k,\Gamma}\varphi_\Gamma)(\x) &:=& \mathrm{f.p.}\int_\Gamma \frac{\partial^2 \phi_k( \x-\y )}{\partial \nnn(\x)\,\partial \nnn(\y) }\varphi_\Gamma(\y) {{\rm d}\y} \nonumber\\
&=&\partial_{\ttt(\x)} \V_{k,\Gamma}[\partial_{\ttt} \varphi_\Gamma](\x)-k^2 \ttt({\x})\cdot (\V_{k,\Gamma}[\ttt \varphi_\Gamma])({\x})\label{eq:W}
\end{eqnarray}
are respectively the single layer, double layer, adjoint double and hypersingular operator.
In the latter operator, ``f.p.'' stands for {\em finite part} since the kernel of the operator is strongly singular. However, as noted, the hypersingular operator can also be written in terms of the single-layer operator and the tangent derivative $\partial_{\ttt}$ operator, an alternative expression which is {sometimes} referred to as Maue's formula.  The following not-so-well-known identities:
\begin{eqnarray}
\label{eq:dt:SL} \partial_{\ttt(\x)}  (\SL_{k,\Gamma}\varphi_\Gamma) (\x) &=&  (\partial_{\ttt} \V_{k,\Gamma}\varphi_\Gamma)(\x),\\
\label{eq:dt:DL} \partial_{\ttt(\x)}  (\DL_{k,\Gamma}\varphi_\Gamma) (\x) &=& \frac12\partial_{\ttt}\varphi(\x) +k^2(\V_{k,\Gamma}(\nnn \varphi_\Gamma))(\x)\cdot\ttt(\x)-
(\K^{\top}\partial_{\ttt}\varphi)(\x),
\end{eqnarray}
again with $\x\in\Gamma$,
 will be used in the formulation of the method. We refer the reader to \cite[Th. 3.5]{DomTurc:2023} for a proof.

\begin{remark}From now on we will denote the layer operators and potentials associated to the Helmholtz equation with a generic wave-number $k$ with the subscript ``$k$''.  When we want to refer  to $k_p$ or $k_s$, the pressure and strain wave-numbers associated to the Navier equation, we will use simply $p$ and $s$ for lighten the notation.
\end{remark}

We propose then a solution of \eqref{eq:Hdecomp2} in the form of a combined potential formulation for $u_p$ and $u_s$:
\begin{equation}\label{eq:up_and_us}
\begin{aligned}
 u_p &= {\DL}_{p,\Gamma}\varphi_{p,\Gamma}-   {\SL}_{p,\Gamma}\mathrm{Y}_{p,\Gamma} \varphi_{p,\Gamma},\quad\\
 u_s &= {\DL}_{s,\Gamma}\varphi_{p,\Gamma}- {\SL}_{s,\Gamma}\mathrm{Y}_{s,\Gamma}
 \varphi_{s,\Gamma},
\end{aligned}
\end{equation}
in terms of densities $\varphi_{p,\Gamma},\ \varphi_{s,\Gamma}$. Here $\mathrm{Y}_{p,\Gamma} $ and $\mathrm{Y}_{s,\Gamma} $ are suitable approximations, which will be described precisely below, of the Dirichlet-to-Neumann operators
$\mathop{\rm DtN}_{p}$ and
$\mathop{\rm DtN}_{s}$
corresponding to exterior Helmholtz problems for $k_p$ and $k_s$ respectively. By construction, $u_p$ and $u_s$ are solutions of the Helmholtz equations satisfying the radiation (Sommerfeld) conditions at infinity. It is worth noting that with the use of such operators, the equation \eqref{eq:up_and_us} resembles the representation formula of the exterior solutions of the Helmholtz equation in terms of the Dirichlet and Neumann data:
\[
u_k = {\DL}_{k,\Gamma}\gamma_{\Gamma} u_k-   {\SL}_{k,\Gamma}  \partial_{\nnn}u_k.
\]

By imposing the Dirichlet conditions in the $(\nnn,\ttt)$ frame on $\Gamma$, and making use of the identities \eqref{eq:traces}--\eqref{eq:dt:DL},  the combined field approach leads to the following system of BIE for the densities $\varphi_{p,\Gamma}$ and $\varphi_{s,\Gamma}$
\[
 \mathcal{A}_{\rm comb,\Gamma} \begin{bmatrix}
                          \varphi_{p,\Gamma}\\
                          \varphi_{s,\Gamma}
                         \end{bmatrix}
=\begin{bmatrix}
   f_{\nnn,\Gamma}\\
   f_{\ttt,\Gamma}
   \end{bmatrix},\quad   f_{\nnn,\Gamma} = -\gamma_{\Gamma}{\bf u}^{\rm inc}\cdot\nnn,\ f_{\ttt,\Gamma} = -\gamma_{\Gamma}{\bf u}^{\rm inc}\cdot\ttt
 \]
 where
 \begin{equation}\label{eq:firstEq}
 \mathcal{A}_{\rm comb,\Gamma}   :=
 \mathcal{A}_{\rm DL,\Gamma} -   \mathcal{A}_{\rm SL,\Gamma}  \begin{bmatrix}
                                         \mathrm{Y}_{p,\Gamma} & \\
                                            & \mathrm{Y}_{s,\Gamma}
                                       \end{bmatrix}
 \end{equation}
with
  \begin{eqnarray}
  \mathcal{A}_{\rm DL,\Gamma}&:=&
\begin{bmatrix}
\W_{p,\Gamma}&\frac{1}2 \partial_{\ttt} + k_{s}^2\ttt \cdot \V_{s,\Gamma}[\nnn\,\cdot\,]-\K_{s,\Gamma}^\top \partial_{\ttt}       \\
\frac{1}2 \partial_{\ttt} + k_{p}^2\ttt \cdot \V_{p,\Gamma}[\nnn\,\cdot\,]-\K_{p,\Gamma}^\top \partial_{\ttt}  &-\W_{s,\Gamma}
\end{bmatrix}\label{eq:ADL}
\\
\mathcal{A}_{\rm SL,\Gamma} &:=&
\begin{bmatrix}
-\frac{1}2 \I + \K_{p,\Gamma}^\top &  \partial_{\ttt} \V_{s,\Gamma}  \\
\partial_{\ttt} \V_{p,\Gamma}  &\frac{1}2 \I - \K_{s,\Gamma}^\top
\end{bmatrix}.\label{eq:ASL}
 \end{eqnarray}

Equation \eqref{eq:firstEq} is unsuitable for numerical approximation, regardless of the choice of the operators $\mathrm{Y}_{p,\Gamma}$ and $\mathrm{Y}_{s,\Gamma}$. Indeed, the operators $\mathcal{A}
_{{\rm comb},\Gamma}: H^{s+1}(\Gamma)\times H^{s+1}(\Gamma)\to H^s(\Gamma) \times H^s(\Gamma)$, where $H^s(\Gamma)$ is the  Sobolev space on $\Gamma$ of order $s$,  although continuous, are not Fredholm operators due to the fact that their kernels and coimages of the principal part are not finite-dimensional. The  root of the problem goes deeper, and can be traced to the Helmholtz boundary conditions themselves. Indeed, the boundary conditions in the $(\nnn,\ttt)$ framework {\eqref{eq:2.7}} can be recast via the exterior Helmnoltz Neumann-to-Dirichlet operators $\mathop{\rm NtD}\nolimits_{k,\Gamma}$ (inverses of the DtN operators) in an alternative form featuring the matrix operator
\begin{equation}\label{eq:def:C}
\mathcal{C}:=
\begin{bmatrix}
\I                                            &  \partial_{\ttt}\mathop{\rm NtD}\nolimits_{s,\Gamma} \\
-\partial_{\ttt}\mathop{\rm NtD}\nolimits_{p,\Gamma} & \I
\end{bmatrix}
: H^s(\Gamma)\times H^s(\Gamma)\to H^s(\Gamma)\times H^s(\Gamma).
\end{equation}
For any Helmholtz exterior solution $u$ we can express
\begin{equation}\label{eq:2.18}
\mathop{\rm NtD}\nolimits_{k,\Gamma}(\partial_{\nnn} u)
= \gamma_{\Gamma}  u
= - 2\V_{k,\Gamma} \partial_{\nnn} u  -2\K_{k,\Gamma} \gamma_{\Gamma} u.
\end{equation}
Besides, if $\V_{0,\Gamma}$ is the Single Layer Operator for the Laplace operator,  it holds  $\K_{k,\Gamma}, \V_{k,\Gamma}-\V_{0,\Gamma} :H^s(\Gamma)\to H^{s+3}(\Gamma)$, i.e. operators of order $-3$. Therefore
\[
 \mathcal{C}= \underbrace{\begin{bmatrix}
              \I                  & -2\partial_{\ttt} \V_{0,\Gamma}\\
              -2\partial_{\ttt}\V_{0,\Gamma}  & -\I\\
              \end{bmatrix}}_{=\mathcal{C}_0} + \mathcal{K}_{2}
\]
with $\mathcal{K}_{2}$ of order $-2$. The principal part of $\mathcal{C}$, $\mathcal{C}_0$, is defective, actually nilpotent, module operators of $-\infty$ order. Indeed
\[
 \mathcal{C}_0^2 = \begin{bmatrix}
              \I + 4\partial_{\ttt} \V_{0,\Gamma}\partial_{\ttt} \V_{0,\Gamma}                 &  \\
                & \I + 4\partial_{\ttt} \V_{0,\Gamma}\partial_{\ttt} \V_{0,\Gamma} 
              \end{bmatrix}= \begin{bmatrix}
               \I               & 0\\
              0&\I
              \end{bmatrix}(\K_{0,\Gamma}^\top)^2
\]
due to the identities
\[
 \partial_{\ttt}\V_{0,\Gamma}\partial_{\ttt} = \W_{0},\quad \W_{0,\Gamma}\V_{0,\Gamma}=-\frac{1}4\I +(\K_{0,\Gamma}^\top)^2.
\] 
(That is, formally \eqref{eq:W} with $k=0$  and Calderon identities for Laplace equation).  $\K^\top_{0,\Gamma}$ is known to be an integral operator with smooth kernel. So
 \[ 
   \mathcal{C}_0^2 : H^s(\Gamma)\times H^s(\Gamma)\to H^r(\Gamma)\times H^r(\Gamma)
 \]
 for any $s,r$ (i.e., a pseudodifferential operator of order $-\infty$).

%
%
%
Naturally, this defective character is inherited by the matrix BIOs $\mathcal{A}_{\rm SL}$ and $\mathcal{A}_{\rm DL}$. Nevertheless, regularizing operators $\mathcal{R}$ can be employed to render the composition $\mathcal{A}_{\rm comb,\Gamma}\mathcal{R}:H^s(\Gamma) \times H^s(\Gamma) \to H^s(\Gamma) \times H^s(\Gamma)$  continuously invertible. For instance, it can be shown that $\mathcal{A}_{\rm comb,\Gamma}^2: H^{s+1}(\Gamma)\times H^{s+1}(\Gamma) \to H^{s+1}(\Gamma)\times H^{s+1}(\Gamma)$ is invertible {\cite{DomTurc:2023,dong2021highly}}, and thus the obvious choice $\mathcal{R}=\mathcal{A}_{\rm comb,\Gamma} (\mathcal{A}_{\rm comb,\Gamma}^2)^{-1}$ could be a regularizing candidate. However, its numerical evaluation becomes too expensive for it to be a viable option in practice. Hence more efficient alternatives have to be considered.

While we postpone the proper definitions of $\mathcal{R}$ and $\mathcal{Y}$ to the next section since they require a principal symbol pseudodifferential calculus, we can provide a general overview of our robust BIE approach. Assuming proper regularizing operators $\mathcal{R}$ are available, our method of solution is outlined below
\begin{enumerate}[label =(\roman*)]
 \item With $f_{\nnn,\Gamma} = -\gamma_{\Gamma}{\bf u}^{\rm inc}\cdot\nnn,\ f_{\ttt,\Gamma} = -\gamma_{\Gamma}{\bf u}^{\rm inc}\cdot\ttt$ find $(\lambda_{p,\Gamma}, \lambda_{s,\Gamma})$ such that
 \begin{equation}\label{eq:AR}
  \mathcal{A}_{\rm comb,\Gamma}\mathcal{R}\begin{bmatrix}
                              \lambda_{p,\Gamma}\\
                              \lambda_{s,\Gamma}
                             \end{bmatrix}
=
\begin{bmatrix}
   f_{\nnn,\Gamma}\\
   f_{\ttt,\Gamma}
   \end{bmatrix}.
 \end{equation}
\item Define
\begin{equation}\label{eq:AR2}
 \begin{bmatrix}
                              \varphi_{p,\Gamma}\\
                              \varphi_{s,\Gamma}
                             \end{bmatrix}=\mathcal{R}\begin{bmatrix}
                              \lambda_{p,\Gamma}\\
                              \lambda_{s,\Gamma}
                             \end{bmatrix}.
\end{equation}
\item Construct ${u}_p$ and $u_s$ according to \eqref{eq:up_and_us}.
\end{enumerate}

In the next section, we will introduce the parameterized Sobolev space which will be essential for the development of the principal symbol Fourier multiplier calculus, which, in turn, allows us to construct a regularizer operator $\mathcal{R}$. We will analyze the resulting combined field equations, we will describe a Nystr\"om discretization for their numerical solution, and we will establish the stability together with the order of convergence of the resulting scheme in the case of arc-length parametrizations. We turn our attention to the case of arbitrary smooth parametrizations, which is considerably more complex, in Section 4.

\section{ Regularized BIEs with arc-length parametrizations}\label{Helm}

We restrict ourselves in this section to work with a regular positive oriented arc-length parametrization of $\Gamma$.
Such an assumption simplifies considerably the construction of the aforementioned regularizing operators, as well as the stability analysis of the Nystr\"om discretizations of the ensuing regularized integral formulations. As we have already mentioned, we return in Section 5 to the case of arbitrary boundary parametrizations.

\subsection{Periodic Sobolev spaces and some useful operators}

Let then ${\bf x} = (x_1(t),x_2(t)):\mathbb{R}\to \Gamma$ be smooth, $2\pi-$periodic parametrization such that
\[
 |{\bf x}'(\tau)| = \sqrt{(x_1'(\tau))^2+(x_2'(\tau))^2} = 1,\quad \forall \tau
\]
The unit tangent and normal parameterized vector to $\Gamma$ (at ${\bf x}(\tau)$) are then given,  see \eqref{eq:2.6}, by
\[
 \ttt(\tau) = {\bf x}'(\tau),\quad
 \nnn(\tau) = \mathrm{Q}{\bf x}'(\tau).
\]
We will identify functions (or distributions) on $\Gamma$, $\varphi_\Gamma: \Gamma\to\mathbb{C}$, with functions, or distributions, on the real line via
\begin{equation}\label{eq:convention}
 \varphi(\tau) = \varphi_\Gamma({\bf x}(\tau)),
\end{equation}
so that
\[
(\partial_{\ttt}\varphi_\Gamma)\circ{\bf x} = \varphi'=: \D\varphi.
\]
Similarly, we denote
\[
(\V_k\varphi)(t)) = \int_0^{2\pi}
\phi_k( {\bf x}(t)-{\bf x}(\tau) )\varphi_\Gamma({\bf x}(\tau))\,
{\rm d}\tau
\]
as the parameterized version of $\V_{k,\Gamma}$. We follow the same convention for the remaining of the BIOs and potentials of Helmholtz Calder\'on calculus.

Sobolev spaces $H^s(\Gamma)$, $s\in\mathbb{R}$, can be then  identified with $2\pi-$periodic Sobolev spaces (see for instance \cite[Ch. 8]{Kress}) given by
\[
 H^s = \Big\{\varphi\in\mathcal{D}'(\mathbb{R}) \ :\  \varphi(\cdot+2\pi) = \varphi, \quad \|\varphi\|_{s}<\infty \Big\}
\]
($\mathcal{D}'(\mathbb{R})$ is  the space of distributions in $\mathbb{R}$). Here
\[
 \|\varphi\|_{s}^2 = |\widehat{\varphi}(0)|^2 + \sum_{n\ne 0} |n|^{2s}
 |\widehat{\varphi}(n)|^2,
\]
is the Sobolev norm of order $s$,
where
\[
\widehat{\varphi}(n) =(\varphi,e_{-n}):=\frac{1}{2\pi}\int_{0}^{2\pi} \varphi(\tau) e_{-n}(\tau)\,{\rm d}\tau,\qquad
 e_{n}(\tau):= \exp (   { i} n \tau  )
\]
is the $n-$th Fourier coefficient. (The integral must be understood in a week sense for non-integrable  functions $\varphi$).
Clearly
\[
 \D\varphi =  i \sum_{n\ne 0}  n \widehat{\varphi}(n) e_{n}.
\]
We will also need integer, positive and negative, powers of $\D$, and the averaging operator $\J$, defined in the following manner
\begin{equation}\label{eq:def:Ds:and:J}
 \D _r\varphi := ( \D  \varphi)^r= \sum_{n\ne 0} (2\pi i n)^r \widehat{\varphi}(n) e_{n}, \quad r\ne 0, \quad \J\varphi := \widehat{\varphi}(0).
\end{equation}
The periodic Hilbert operator
\begin{equation}\label{eq:def:H}
 \HH \varphi:=
i \bigg[ \widehat{\varphi}(0)+\sum_{n\ne 0}\mathop{\rm \rm sign}(n)\widehat{\varphi}(n)e_{n}(\cdot){\cdot})\bigg]  = \mathrm{p.v.}\frac{1}{2\pi}\int_0^{2\pi}  \cot\frac{\cdot-\tau}2 \varphi(\tau)\,{\rm d}\tau + \frac{i}{2\pi}\int_{0}^{2\pi} \varphi(\tau)\,{\rm d}\tau
\end{equation}
(``p.v.'' stands for the Cauchy  principal value; obviously the integral has to be understood in a distributional sense for non integrable functions.)
will also play  an essential role in what follows.
We will refer to these operators as Fourier multipliers since they are diagonal in the Fourier basis of complex exponentials $\{e_n\}_{n\in\mathbb{Z}}$.

Clearly, $\D_r:H^s\to H^{s-r}$ and $\HH:H^s\to H^s$ are continuous for any $s$. We then say that $\D_r$ and $\HH$ are pseudodifferential operators of order $r$ and $0$, respectively and will write
\[
 \D_r\in \mathrm{OPS}(r),\quad
 \HH\in \mathrm{OPS}(0).
\]
If $\mathrm{A}\in \mathrm{OPS}(r)$ for any $r$, such as the mean value operator  $\J$ introduced above, we simply write $\mathrm{A}\in\mathrm{OPS}(-\infty)$. It is a well-established result that $\mathrm{OPS}(-\infty)$ can be identified with periodic integral operators with smooth kernels \cite{hsiao2008boundary,Saranen}.

These notation will be extended to matrix operator in such a way that we write
\[
 \mathcal{A} =\begin{bmatrix}
            A_{11}&A_{12}\\
            A_{21}&A_{22}
           \end{bmatrix}\in\mathrm{OPS}(r)
\]
if $A_{ij}\in\mathrm{OPS}(r)$. Similarly, we will say that $\mathcal{A}$ is a Fourier {multiplier} operator if so are $A_{ij}$.

\subsection{Parameterized Boundary Integral Operators}

Let us introduce for $r=0,1,2,\ldots$
\begin{equation}  \label{eq:def:rhor}
\begin{aligned}
  \Lambda_r\varphi \ :=\  &\frac{1}{2\pi}\int_0^{2\pi}\rho_r(\replaced[id = vD]{\cdot}{t}-\tau)\varphi(\tau)\,{\rm d}\tau=\sum_{n=-\infty}^\infty \widehat{\rho}_r(n)\varphi(n)e_n, \\
  &\qquad \qquad \rho_r(\tau):=-(e_1(\tau)-1)^{r-1} \log\left(2\left|\sin \frac{ \tau}2\right |\right).
\end{aligned}
 \end{equation}


We point out that
\begin{equation}\label{eq:HD-1}
\Lambda_1 \varphi= \frac12
 \HH\D_{-1} \varphi
\end{equation}
which is equivalent to stating that
\[
 \widehat{\rho}_1(n)=\begin{cases}
             \frac{1}{2|n|} , &n\ne 0, \\[2ex]
            0 , & n=0.
           \end{cases}
\]
Similarly, it can be easily proven from the relation
\[
 \widehat{\rho}_r(n) = \widehat{\rho}_{r-1}(n)- \widehat{\rho}_{r-1}(n-1)=\sum_{j=0}^{r-1} \binom{r-1}j (-1)^{j-r+1}\widehat{\rho}_1(n-j)
\]
that
\[
 \widehat{\rho}_r(n) = \frac{1}{2r} \binom{n}r ^{-1}\mathop{\rm sign}(n),\quad n\ne 0,1,2,\ldots,r.
\]
%
More specifically,
\begin{equation}\label{eq:fourier:rho}
\begin{aligned}
 \widehat{\rho}_2(n)& =\begin{cases}
            \frac12 , & n=0,\\
             -\frac12 , & n=1,\\
            \frac{\sign n}{2 n(n-1) }, &n\ne 0,1,\\
           \end{cases} \quad
 \widehat{\rho}_3(n)=\begin{cases}
            -\frac34 , & n=0,2,\\
             1 , & n=1,\\
            \frac{1}{ |n(n-1)(n-2)| }, &n\ne 0,1,2,\\
           \end{cases}  \\
 \widehat{\rho}_4(n)&=\begin{cases}
             \frac{11}{12} , & n=0,\\
             -\frac{7}{4}  , & n=1,\\
              \frac{7}{4}  , & n=2,\\
            -\frac{11}{12}  , & n=3,\\
            \frac{3\sign n}{ n(n-1)(n-2)(n-3) }, &n\ne 0,1,2.\\
           \end{cases}  \quad
\end{aligned}
\end{equation}
\begin{lemma}\label{lemma:2.2}
 It holds
 \begin{eqnarray*}
  \V_k &=& \V^{(1)}_k \\
  &=& \Lambda_1  + \V_{k}^{(3)}=\frac{1}2\HH\D_{-1}  + \V_{k}^{(3)} \\
   &=&  \Lambda_1 + \frac{k^2}4\Lambda_3+\V_{k}^{(4)}=
   \frac{1}2\HH\D_{-1}-  \frac{k^2}4\HH\D_{-3} + \frac{k^2}4\underbrace{(\Lambda_3+\HH\D_{-3})}_{\in\mathrm{OPS}(-5)}+\V_{k}^{(4)}
 \end{eqnarray*}
 where, for $j=1{,}\  3,\  4$, $\V_{k}^{(j)} \in\mathrm{OPS}(-j)$ are integral operators which can be written in terms of periodic integral operators with explicit kernels as
 \begin{equation}\label{eq:01:lemma:2.2}
\V_{k}^{(j)} \varphi := \int_{0}^{2\pi } A^{(j)}(\cdot,\tau)
(e_1(\replaced[id = vD]{\cdot}{t}-\tau)-1)^{ j-1}
\log \left(4\sin^2\frac{\cdot -\tau}2\right) \varphi(\tau)\,{\rm d}\tau +\int_{0}^{2\pi } B^{(j)}(\cdot,\tau) \varphi(\tau)\,{\rm d}\tau
 \end{equation}
 where
$A^{(j)}$ and $B^{(j)}$ are smooth bi-periodic functions.
\end{lemma}
\begin{proof} The proof is based on the decomposition of the fundamental solution of the Helmholtz equation
\[
 \phi_k(\bm{x})  = -\frac{1}{4\pi} J_0(k|\bm{x}|)\log|\bm{x}|^2+C(|{\bm x}|)
\]
with $J_0$ the Bessel function of first kind and order zero and $C$ a smooth function.
We refer the reader to \cite[\S 10.4]{hsiao2008boundary} and \cite{DomTurc:2023} or  \cite[Ch. 12]{Kress} for the decomposition of the integral operator.
 It is a well-established result that integral operators of this kind define operators of order $-j$ (see, for example,  \cite[Ch. 7]{Saranen}).
\end{proof}

As consequence, we derive the following expansion for the single layer operator
\[
   \V_k = \frac{1}2\HH\D_{-1}-  \frac{k^2}4\HH\D_{-3} +\mathrm{OPS}(-4).
\]
Actually, it can be shown with little effort that the remainder is of order $-5$. However, we do not use this additional regularity order.

\begin{lemma}\label{lemma:2.3}
 It holds
 \begin{eqnarray*}
  \W_k &=& \frac{1}2\HH\D+\frac{k^2}4\HH\D_{-1} +\underbrace{
\frac{k^2}4(\D\Lambda_3\D+ \HH\D_{-1})}_{\in\mathrm{OPS}(-3)}+ \W_k^{(2)}
 \end{eqnarray*}
 where
 \[
\W_{k}^{(2)}  :=  \D\V_{k}^{(4)} \D +k^2\widetilde{\V}_{k}^{(3)} \in\mathrm{OPS}(-2)
 \]
 with
 \begin{equation}\label{eq:02:lemma:2.3}
  \widetilde{\V}_{k}^{(3)}{\varphi} = \int_{0}^{2\pi } \widetilde{A}^{(3)}(\cdot,\tau)
  (e_1(\cdot-\tau)-1)^2\log \left(4\sin^2 \frac{\cdot -\tau}2\right) \varphi(\tau)\,{\rm d}\tau +\int_{0}^{2\pi } \widetilde{B}^{(3)}(\cdot,\tau) \varphi(\tau)\,{\rm d}\tau
 \end{equation}
and
 \begin{eqnarray*}
 \widetilde{A}^{(3)}(t,\tau) &:=&  A^{(3)}(t,\tau)\bm{t}(t)\cdot\bm{t} (\tau)+
  \frac{k^2}{4\pi  (e_1(t-\tau)-1)^2}(1- \ttt(t)\cdot \ttt(\tau) ),\\
 \widetilde{B}^{(3)}(t,\tau) &:=&B^{(3)}(t,\tau)\bm{t}(t)\cdot\bm{t} (\tau).
 \end{eqnarray*}
 are smooth and biperiodic.
\end{lemma}
\begin{proof} It is consequence of the Maue's formula \eqref{eq:W} which in this context can be written as
\begin{eqnarray*}
 \W_k &=& \D \V_k\D +k^2 \ttt\cdot \V_k[\ttt\cdot] \\
 &&=
 \frac{1}2\HH\D+\frac{k^2}4\HH\D_{-1}+ \D\V_{k}^{(4)} \D +
 \ttt\cdot \V_{k}^{(3)}[\ttt\cdot]+\frac{k^2}2 \left(\ttt \cdot \HH\D_{-1}[\ttt \cdot]-\HH\D_{-1} \right){,}
\end{eqnarray*}
with \eqref{eq:def:rhor} and Lemma \ref{lemma:2.2}.
\end{proof}

Notice that
\[
 \widetilde{A}^{(3)}(t,t) = A^{(3)}(t,t) -\frac{k^2}{8\pi} \kappa^2(t)
\]
where  $\kappa$ is the signed curvature given by
\begin{equation}\label{eq:def:kappa}
  \kappa = -\nnn\cdot \partial_{\ttt} \ttt =|{\bf x}''|.
 \end{equation}

\begin{lemma}\label{lemma:2.4} The   adjoint double layer operator and double layer operator can be written as
\begin{eqnarray*}
\mathrm{K}_{k}^\top    \varphi &:=& \int_{0}^{2\pi } C(\cdot,\tau)  (e_1(\cdot-\tau)-1)^2\log \left(4 \sin^2 \frac{\cdot -\tau}2\right) \varphi(\tau)\,{\rm d}\tau +\int_{0}^{2\pi } D(\cdot,\tau) \varphi(\tau)\,{\rm d}\tau,
 \\
 \mathrm{K}_{k} \varphi &:=& \int_{0}^{2\pi } C(\tau,\cdot)
 (e_1( \tau-\cdot )-1)^2\log \left(4 \sin^2 \frac{\cdot -\tau}2\right) \varphi(\tau)\,{\rm d}\tau +\int_{0}^{2\pi } D(\tau,\cdot) \varphi(\tau)\,{\rm d}\tau
\end{eqnarray*}
with $C$ and $D$ smooth, which implies that $\K,\ \K^\top\in\mathrm{OPS}(-3)$.
\end{lemma}

Finally, we analyze the integral operator
\begin{eqnarray*}
 [{\V}_{k,\ttt,\nnn} \varphi](t)  &=&  \ttt(t)\cdot \V_{k}[\nnn\varphi](t){.}
 \end{eqnarray*}

 \begin{lemma}\label{lemma:2.5}
  It holds
 \begin{equation}\label{eq:01:lemma:2.5}
{\V}_{k,\ttt,\nnn}  \varphi  =  -\frac{1}{2\pi}  \int_{0}^{2\pi} {A}_{\ttt}(\cdot,\tau)(e_1(\cdot-\tau)-1)\log\left(4\sin^2\frac{\cdot-\tau}2\right) \varphi(\tau)\,{\rm d}\tau
+\int_{0}^{2\pi} {B}_{\ttt} (\cdot,\tau)  \varphi(\tau)\,{\rm d}\tau
 \end{equation}
where
\begin{eqnarray*}
{A}_{\ttt}(t,\tau) &:=&
\frac{1}{e_1(t-\tau)-1} (\ttt(t)\cdot \nnn(\tau))
A^{(1)}(t,\tau), \\
{B}_{\ttt}(t,\tau) &:=& B^{(1)}(t,\tau)(\ttt(t)\cdot \nnn(\tau)),
\end{eqnarray*}
with $A^{(1)}$ and $B^{(1)}$ as in Lemma \ref{lemma:2.2}. Furthermore, ${\V}_{k,\ttt,\nnn} \in\mathrm{OPS}(-2)$.
 \end{lemma}
\begin{proof}
 Straightforward from the definition, Lemma \ref{lemma:2.3} and \eqref{eq:HD-1}.
\end{proof}


Let us point out that in this case
\[
   {A}_{\ttt}(t,t) = -\frac{i}{2\pi}\kappa(t).
\]

\subsection{Arc-length parameterized regularized integral equations}

We are now in the position to construct the DtN {(Dirichlet-to-Neumann)} approximations as well as the regularizing operators featured in equations~\eqref{eq:AR}.

In view of the results of the previous subsection, we have that the (arclength-) parameterized versions of the operators    $\mathcal{A}_{\rm DL,\Gamma},\ \mathcal{A}_{\rm SL,\Gamma}$ are given by
  \begin{equation}\label{eq:ADLASL}
  \begin{aligned}
  \mathcal{A}_{\rm DL}\ :=\ &     \frac{1}2\mathcal{H}_0 \HH\D+\frac{1}4
  \begin{bmatrix}
                                         k_{p}^2 & \\
                                            & -k_{s}
                                             \end{bmatrix}\HH\D_{-1}
\
+
\frac14
\begin{bmatrix}
      k_p^2                     &\\
&- k_s^2     \\
    \end{bmatrix}(\D\Lambda_3\D+\HH\D_{-1})
\\
 &\
+ \begin{bmatrix}
     \W^{(2)}_{p}                      &k_s^2 {\V}_{s,\ttt,\nnn}-\K_s^\top \D\\
     k_p^2 {\V}_{p,\ttt,\nnn}-\K^\top_p\D & -\W^{(2)}_{s}\\
    \end{bmatrix},\\
  \mathcal{A}_{\rm SL}\ :=\ & \frac{1}2\mathcal{H}_0
    +\frac{1}4\begin{bmatrix}
                                                    &k_{p}^2 \\
                                              k_{s}^2 &
                                       \end{bmatrix}\HH\D_{-2}
     + \frac14
    \begin{bmatrix}
     &k_s^2\\
     k_p^2&
    \end{bmatrix}
(\Lambda_3+\HH\D_{-3})\D
\\
 &     -
    \begin{bmatrix}
     \K^\top _{p} & \D \V_{s}^{(4)}\\
    \D \V_{p}^{(4)} & - \K^\top_{s}
    \end{bmatrix}
   .
  \end{aligned}
   \end{equation}
In the expressions above
 \begin{equation}\label{eq:defH0}
     \mathcal{H}_0 :=  \begin{bmatrix}
 \mathrm{I}&-\mathrm{H} \\
 -\mathrm{H}&-\mathrm{I}
 \end{bmatrix}
 \end{equation}
 that satisfies $\mathcal{H}_0^2 = 0$ (see the discussion below \eqref{eq:def:C}). That is, the principal part for both $\mathcal{A}_{\rm DL}\in\mathrm{OPS}(1)$ and $\mathcal{A}_{\rm SL}\in\mathrm{OPS}(0)$ is defective.

We turn our analysis now to the DtN approximation operator. The representation formula for the exterior Helmholtz solutions \eqref{eq:2.18} and the jump relations for the boundary layer potentials \eqref{eq:traces} imply
\[
\mathrm{DtN}_k=  2\mathrm{W}_k -2\mathrm{K}_k^\top \mathrm{DtN}_k = 2\mathrm{W}_k +\mathrm{OPS}(-2).
\]
This supports the idea of replacing $\mathrm{DtN}_k$ in regularizer formulations by the hypersingular operator  cf. \cite{MR3343368}. The real wave-number $k$ is typically replaced by  the complexified wavenumber $\widetilde{k} = k+ i\epsilon$ in the definition of approximations to DtN operators in order to ensure the invertibility of the ensuing regularized operators. We follow this idea and, in view of Lemma \ref{lemma:2.3}, we introduce {(see \eqref{eq:def:Ds:and:J} and \eqref{eq:def:H})}
\begin{equation}\label{eq:Yk}
 \mathrm{Y}_k =  \HH\D +\frac{\widetilde{k}^2}2 \HH\D_{-1}
 +\J
\end{equation}
with $\widetilde{k} = k+ i\epsilon$ and $\epsilon>0$. Clearly
\[
 \mathrm{Y}_k: H^s\to H^{s+1}
\]
is invertible, as a Fourier multiplier operator  it suffices to examine its action on $\{e_n\}_{n\in\mathbb{Z}}$,
i.e, $\mathrm{Y}_k\in\mathrm{OPS}(1)$ and
\begin{equation}\label{eq:Yk-2Wk}
\mathrm{Y}_k-2\mathrm{W}_k,\
\mathrm{Y}_k- 2\mathrm{DtN}_k \in \mathrm{OPS}(-1).
\end{equation}
The choice
\[
\mathcal{Y} = \begin{bmatrix}
               \mathrm{\mathrm{Y}_{p}}&\\
                      & \mathrm{Y}_{s}
              \end{bmatrix}
\]
in the representations~\eqref{eq:up_and_us}-\eqref{eq:firstEq}  is then fully justified.

The regularizer operator we will use is suggested by the first  terms in the definition of $\mathcal{A}_{\rm DL}$ and $\mathcal{A}_{\rm SL}$ such as it is presented in
\eqref{eq:ADLASL}, namely
 \begin{equation}\label{eq:R}
  \mathcal{R}:=
\mathcal{H}_0\HH\D
+\frac{1}{2}\begin{bmatrix}
             \widetilde{k}_s^2&\\
                            & -\widetilde{k}_p^2
            \end{bmatrix}\HH\D_{-1}
            +\begin{bmatrix}
             \J&\\
                            &\J
            \end{bmatrix}.
 \end{equation}
 Clearly, $\mathcal{R}\in\mathrm{OPS}(1)$. This choice ensures that $\mathcal{R}$ is an approximate right zero divisor of the principal part of the operator $\mathcal{A}_{\rm comb}$, as will be seen in the following theorem. We refer the reader to \cite{DomTurc:2023} for a thorough analysis of this kind of regularizers. Notice in passing that  the principal part of $\mathcal{R}$ is again  nilpotent:
 Hence
 \[
  \mathcal{R}^2=  \frac12\begin{bmatrix}
 2  \widetilde{k}_s^2\mathrm{I}& (\widetilde{k}_s^2-\widetilde{k}_p^2)\mathrm{H} \\
  (\widetilde{k}_s^2-\widetilde{k}_p^2)\mathrm{H}& 2\widetilde{k}_p^2\mathrm{I}
 \end{bmatrix}(\mathrm{I}-\mathrm{J})
+\begin{bmatrix}
            \J&  \\
                        &     \J
            \end{bmatrix}
            \in\mathrm{OPS}(0)
 \]
 which can be easily seen to be invertible with inverse
 \[
  \frac{2}{(\widetilde{k}_p^2+\widetilde{k}_s^2)^2} \begin{bmatrix}
 2  \widetilde{k}_s^2\mathrm{I}& -(\widetilde{k}_s^2-\widetilde{k}_p^2)\mathrm{H} \\
 - (\widetilde{k}_s^2-\widetilde{k}_p^2)\mathrm{H}& 2\widetilde{k}_p^2\mathrm{I}
 \end{bmatrix}(\mathrm{I}-\mathrm{J})
+\begin{bmatrix}
            \J&  \\
                        &     \J
            \end{bmatrix}.
 \]


%
%
With these choices of regularizing operators, we are now in the position to analyze the ensuing regularized formulations featuring the operator composition $\mathcal{A}_{\rm comb} \cal{R}$. To this end, we first introduce the decomposition
\begin{equation}\label{eq:Acomb}\begin{aligned}
\mathcal{A}_{\rm comb}  \ =\ & \underbrace{ \left( \frac{1}2\mathcal{H}_0 \HH\D+\frac{1}4  \begin{bmatrix}
                                         k_{p}^2 & \\
                                            & -k_{s}
                                       \end{bmatrix}\HH\D_{-1}
    +\frac{1}2\mathcal{H}_0\mathcal{Y}
    +\frac{1}4\begin{bmatrix}
                                                    &k_{p}^2 \\
                                              k_{s}^2 &
                                       \end{bmatrix}\HH\D_{-2}\mathcal{Y}
    \right)}_{\mathcal{A}_{\rm comb,pp}^{(-1)}}\\
&\
+
\underbrace{
\frac14\begin{bmatrix}
      k_p^2                     &\\
&- k_s^2     \\
    \end{bmatrix}(\D\Lambda_3\D+\HH\D_{-1})
    +\frac14
    \begin{bmatrix}
     &k_s^2\\
     k_p^2&
    \end{bmatrix}
\mathcal{Y }(\Lambda_3+\HH\D_{-3})\D
}_{\mathcal{A}_{\rm comb,pp}^{(2)}}
\\
&\   \underbrace{
+
\left(\begin{bmatrix}
     \W^{(2)}_{p}                      &k_s^2 {\V}_{s,\ttt,\nnn}-\K_s^\top \D\\
     k_p^2 {\V}_{p,\ttt,\nnn}-\K^\top_p\D & -\W^{(2)}_{s}\\
    \end{bmatrix}-
    \begin{bmatrix}
     \K^\top _{p} & \D \V_{s}^{(4)}\\
    \D \V_{p}^{(4)} & - \K^\top_{s}
    \end{bmatrix}\mathcal{Y}\right)}_{\mathcal{A}_{\rm comb}^{(2)}}
   .
\end{aligned}
\end{equation}
We stress that $\mathcal{A}_{\rm comb,pp}^{(-1)}\in\mathrm{OPS}(1)$ and $\mathcal{A}_{\rm comb,pp}^{(2)}\in\mathrm{OPS}(-2)$ are Fourier multiplier operators\deleted[id=vD]{, and therefore maps $\mathbb{T}_n\times \mathbb{T}_n$ into itself}. Furthermore, $\mathcal{A}_{\rm comb}^{(2)}\in\mathrm{OPS}(-2)$.
\begin{theorem}\label{theo:01}
It holds
\begin{eqnarray*}
 \mathcal{A}_{\tt comb}\mathcal{R}
 &=& -\frac{1}4\begin{bmatrix}
  \left(k_p^2+k_s^2+\widetilde{k}_p^2+ \widetilde{k}_s^2\right)\mathrm{I}&
    -\left(k_p^2+k_s^2-\widetilde{k}_p^2- \widetilde{k}_s^2\right)\mathrm{H} \\
    \left(k_p^2+k_s^2-\widetilde{k}_p^2- \widetilde{k}_s^2\right)\mathrm{H}&
   \left(k_p^2+k_s^2+\widetilde{k}_p^2+ \widetilde{k}_s^2\right)\mathrm{I}
\end{bmatrix}+\mathrm{OPS}(-1)
\end{eqnarray*}
and both the main part as well as the operator $\mathcal{A}_{\tt comb}\mathcal{R}$ itself are invertible.
\end{theorem}
\begin{proof} We refer the reader to \cite{DomTurc:2023}.  The decomposition can be derived from straightforward calculations on the term $
 \mathcal{A}_{\rm comb,pp}^{(-1)}\mathcal{R}$ {as it is done in \eqref{eq:Hps}}.
\end{proof}

\begin{remark}
From previous Theorem, the identity
\[
 \begin{bmatrix}
  \varphi_p\\
\varphi_s
 \end{bmatrix}={\cal R}({\cal A}_{\rm comb}{\cal R})^{-1} \begin{bmatrix}
  f_{\nnn}\\
f_{\ttt}
 \end{bmatrix},
\]
the fact that ${\cal R}, \mathcal{Y}\in \mathrm{OPS}(1)$, the continuity of the trace operator $H^r(\Omega)\to H^{r-1/2}(\Gamma)$,
the continuity of the boundary layer potentials ${\rm DL}_{k,\Gamma}:H^s(\Gamma)\to H^{s+1/2}_{\rm loc}(\Omega^+)$, ${\rm SL}_{k,\Gamma}:H^{s-1}(\Gamma)\to H^{s+1/2}_{\rm loc}(\Omega^+)$ (for $s\ge 0$; see \cite{doi:10.1137/0519043} or \cite[Ch. 6 \& 8]{mclean:2000} and references therein for the limit case $s=0$), we can infer that for $r>3/2$ there exists $C>0$,
\[
\|u_p\|_{H^{r-1}_{\rm loc}(\Omega_+)}+
\|u_s\|_{H^{r-1}_{\rm loc}(\Omega_+)}\le C \|\bm{u}\|_{H^{r}_{\rm loc}(\Omega^+)\times H^{r}_{\rm loc}(\Omega^+)}
\]
which can be understood as a result on the stability of the Helmholtz decomposition for the solution of the Dirichlet problem for Navier equations.
\end{remark}


%
%
%
%
%
%
%
%

\section{Nystr\"om discretization for arc-length parametrizations BIEs and error analysis}
\label{err_anal}

In this section, we introduce a complete discrete Nystr\"om method to solve \eqref{eq:AR}. Then, we prove stability and convergence in Sobolev norms.

\subsection{ Nystr\"om method}
Let  $N$ be a positive integer and  denote the discrete space
\begin{equation}\label{eq:TN}
 \mathbb{T}_N = \mathop{\rm span}\langle e_n\ :\ n\in\mathbb{Z},\  -N/2\le n< N/2 \rangle.
\end{equation}
On $\mathbb{T}_N$ we consider the interpolation operator
\[
 \mathbb{T}_N\ni Q_N f\quad \text{s.t.}\quad (Q_Nf)(mh) = f(mh),\quad m\in\mathbb{Z}
\]
where $\{mh\}_{m\in\mathbb{Z}}$ are the grid points with mesh size $h:=2\pi/N$.

The Nystr\"om discretization we propose is basically a projection method in the space $\mathbb{T}_N\times \mathbb{T}_N$. Hence the action of the operator $\mathcal{A}_{\rm comb}\mathcal{R}$ on the elements of the space $\mathbb{T}_N$ must be either   explicitly available  or sufficiently well approximated. In the first case, we have $ \mathcal{A}_{\rm comb,\ pp}^{(1)}$,  $ \mathcal{A}_{\rm comb,\ pp}^{(-1)}$,  $\mathcal{R}$ and $\mathcal{Y}$, due to the fact the four of them are Fourier multiplier operators and therefore act  $\mathbb{T}_N\times  \mathbb{T}_N$ to itself. Hence, the design of a Nystr\"om discretization hinges on the construction of sufficiently accurate approximations
\[
\mathcal{A}_{{\rm comb}, N}^{(2)}\approx
\mathcal{A}_{{\rm comb}}^{(2)} =
\begin{bmatrix}
     \D  \V_{p}^{(4)} \D +\widetilde{\V}_{p }^{(3)}                     &k_s^2   {\V}_{s,\ttt,\nnn}
      -\K_{s}^\top \D\\
     k_p^2   {\V}_{p,\ttt,\nnn} -\K_{p}^\top \D\ & -
     \D  \V_{s}^{(4)} \D +\widetilde{\V}_{s }^{(3)} \\
    \end{bmatrix}-
    \begin{bmatrix}
     \K^\top _{p} & \D  \V_{s}^{(4)}\\
    \D  \V_{p}^{(4)} & - \K^\top_{s}
    \end{bmatrix}\begin{bmatrix}
               \mathrm{Y}_p &\\
                      & \mathrm{Y}_{s}
              \end{bmatrix}.
\]
In other words, we have to describe which quadrature rules are going to be used  in the approximations of the different integrals arising in the operators presented in the matrix
$\mathcal{A}^{(2)}_{{\rm comb}}$. For these purposes we consider singular product integration rules,  first introduced in \cite{Kress} and widely
used since then, which take care of
the different logarithmic singularities present in the definition of the operator $\mathcal{A}_{\rm comb}^{(2)}$.
In addition, we also have to consider the action of the derivative operator
when acting on elements which are not in $\mathbb{T}_N$.

In short, the following operators must be approximated via singular quadratures
\begin{equation}
\begin{aligned}
\D\V^{(4)}_k\D, & &  \widetilde{\V}_k^{(3)}, & &  {\V}_{k,\ttt,\nnn},&& 
 \K_k^\top\D,& &\K_k^\top,&& \D\V_k^{(4)}.
\end{aligned}
\end{equation}
We will describe in what follows such approximations. In view of Lemma \ref{lemma:2.4} we see that the approximation of  the operators $\K_k^\top$ (recall that these operators are in $\mathrm{OPS}(-3)$) involves dealing with integral operators of the following form
\[
(\K_k^\top \varphi)(t)=\int_{0}^{2\pi} C(t,\tau)(e_1(\tau-t)-1)^2\log\left(4\sin^2\frac{t-\tau}2\right)\varphi(\tau)\,{\rm d}\tau +
\int_{0}^{2\pi} D(t,\tau)\varphi(\tau)\,{\rm d}\tau.
\]
We then define the semi-discrete approximations
\begin{equation}\label{eq:log01}
(\K_{k,N}^\top \varphi)(t):
 =\int_{0}^{2\pi} (e_1(t-\tau)-1)^2\log\left(4\sin^2\frac{t-\tau}2\right)
 (Q_N(C(t,\cdot)\varphi)(\tau)\,{\rm d}\tau +
\int_{0}^{2\pi} Q_N (D(t,\cdot)\varphi)(\tau)\,{\rm d}\tau.
\end{equation}
Clearly, in the case of a regular kernel $D(t,\tau)$ we have
\[
 \int_{0}^{2\pi} Q_N (D(t,\cdot)\varphi)(\tau)\,{\rm d}\tau = \frac{\pi}{2N}
 \sum_{j=0}^{N-1}
 D(t,t_j)\varphi(t_j)
\]
which is a simple application of the trapezoidal rule  to the underlying integral. Furthermore, singular quadratures for the first integral operator in the right hand side of equations~\eqref{eq:log01} can be easily derived from the identity, see \eqref{eq:fourier:rho},
\[
-\frac{1}{4\pi}\int_{0}^{2\pi} (e_{1}(\tau)-1)^2\log\left(4\sin^2\frac{\tau}2\right) e_{-n}(\tau)\,{\rm d}\tau= \widehat{\rho}_3(n)
\]
The operator $\mathrm{K}_k^\top \D$ is approximated in the same way, i.e.,
\begin{equation}\label{eq:log01D}
 \mathrm{K}_{k,N}^\top \D\approx \mathrm{K}_k^\top \D
\end{equation}
since as we will see later, the derivative operator is  applied formally only on  trigonometric polynomials in which case it can be computed exactly.

Furthermore, the same technique as in~\eqref{eq:log01} is applied to discretize the operator $\widetilde{\V}_{k}^{(3)}$, which appears in the  regular part of the hypersingular operator $\W_k^{(2)}$. Similarly, see Lemma \ref{lemma:2.2},
\begin{equation}
\begin{aligned}
   [\V_{k,N}^{(4)}\varphi](t) \ := &
 \int_{0}^{2\pi} (e_1(t-\tau)-1)^3\log\left(4\sin^2\frac{t-\tau}2\right)
 (Q_N(A^{(4)}(t,\cdot)\varphi)(\tau)\,{\rm d}\tau \\
  &\ +
\int_{0}^{2\pi} Q_N (B^{(4)}(t,\cdot)\varphi)(\tau)\,{\rm d}\tau
\approx  [ \V_{k}^{(4)}\varphi](t)
\end{aligned}
\end{equation}
which can be computed explicitly using the Fourier coefficients of $\rho_4$ (see again \eqref{eq:fourier:rho}) so that we can introduce
\begin{equation}\label{eq:Wkn}
\W_{k,N}^{(2)} :=  \D Q_N\V_{k,N}^{(4)} \D +\widetilde{\V}_{k,N }^{(3)}\approx
\D \V_{k}^{(4)} \D +\widetilde{\V}_{k }^{(3)} = \W^{(2)}_{k}.
\end{equation}
Observe that the leftmost derivative operator has undergone the approximation of $\D\approx \D Q_N$, but such an approximation is not needed for the rightmost one.
%

Similarly, we construct
\[
 \D Q_N\V_{k,N}^{(4)} \approx \V_{k}^{(4)}.
\]
The last   remaining  operator  is $ {[}{\V}_{k,\ttt,\nnn} \varphi{]}(t)= \ttt(t)\cdot \V_{k}[\nnn\varphi](t)$ that can be dealt with in a similar manner
\begin{equation}\label{eq:log03}
 \begin{aligned}
  [ \V_{k,\ttt,\nnn,N} \varphi](t)\ :=\  & \frac{1}{2\pi}  \int_{0}^{2\pi}
  Q_N({A}_{\ttt}(\cdot,\tau)\varphi)(\tau) (e_1(\tau)-1)\log\left(4\sin^2\frac{t-\tau}2\right) \,{\rm d}\tau
\\
&\ +\int_{0}^{2\pi} Q_N(  {B}_{\ttt}(t,\cdot)\deleted[id=vD]{(\tau)}  \varphi)(\tau)\,{\rm d}\tau \approx [\V_{k,\ttt,\nnn} \varphi](t)
\end{aligned}
 \end{equation}
and whose evaluation is carried out with the help of the expressions $\widehat{\rho}_2(n)$ {(see again \eqref{eq:fourier:rho})}.
%

\begin{remark}\label{remark:the:remark}Notice that all the discretizations above follow the same prototype.
\begin{enumerate}[label={\rm (\alph*)}]
\item First, we have the {\em continuous} operator:
\begin{multline*}
(\mathrm{A}_\delta \varphi)(t) =\frac{1}{2\pi}\int_{0}^{2\pi} A(t,\tau)\delta(t-\tau)\varphi(\tau)\,{\rm d}\tau
= \sum_{n\in\mathbb{Z}} \widehat{\delta}(n)\varphi_A(t;n)  e_n(t)
\\
\varphi_A(t;n) = \widehat{[A(t,\cdot)\varphi]}(n) = \frac{1}{2\pi}\int_0^{2\pi} A(t,\tau)\varphi(\tau)e_{-n}(\tau)\,{\rm d}\tau.
\end{multline*}
\item Second, its numerical approximation is constructed, with $\nu =\lfloor n/2\rfloor$, via
\begin{multline*}
 (\mathrm{A}_{\delta,N}\varphi)(t) =\frac{1}{2\pi}\int_{0}^{2\pi} Q_N(A(t,\cdot)\varphi)(\tau) \delta(t-\tau) \,{\rm d}\tau
 = \sum_{n=-\nu}^{N-\nu-1} \widehat{\delta}(n) \varphi_{A,N}(t;n) e_n(t)\\
  \varphi_{A,N}(t,n)= \widehat{[Q_N(A(t,\cdot)\varphi]}(n)=
  \frac{1}{2\pi}\int_0^{2\pi} Q_N(A(t,\cdot)\varphi)(\tau)e_{-n}(\tau)\,{\rm d}\tau.
\end{multline*}
\end{enumerate}
In these expressions,  the Fourier coefficients $(\widehat\delta(n))_n$ are explicitly known which makes possible the exact evaluation of $\mathrm{A}_{\delta,N}$. Furthermore,   the trigonometric interpolating polynomial $Q_N$ or the discrete (Fourier) approximation of the derivative $\D Q_N\approx \D$  can be fit into this frame  with $A\equiv 1$ and $\delta = 2\pi \delta_0$,  $\delta_0$ being the Dirac delta at zero,    in the first case and
$\delta = 2\pi \delta'_0$ in the second case. This explains why we will be able to  deal with apparently different convergence estimates in the same manner.

Besides, it actually holds
  \[
\varphi_{A,N}(t,n) = h \sum_{\ell=0}^{N-1} (A(t, \cdot)\varphi)(\ell h) e_{-n}(\ell h), \quad n = -\nu,\ldots,N-1-\nu,\ h =\frac{2\pi}{N},
\]
which is just consequence of exactness of the rectangular rule in $\mathbb{T}_{2N-1}$. In particular, it makes FFT techniques directly applicable in the implementation. 

Hence,  when the action of the operators are evaluated at the node points $\{j h\}_{j\in\mathbb{Z}}$, such as the Nystr\"om method requires, we reduce this calculation to the matrix-vector product
\[
 {\bm\psi}_N = \bm{A}_{\delta,N}\bm{\varphi}_N, \quad \text{as approximation of $\psi = {\rm A}_{\delta}\varphi$}
\]
where 
\[
\begin{aligned}
{\bm\psi}_N&\approx (\psi(\ell h))_{\ell =0,\ldots,N-1}\in\mathbb{C}^N,
\quad
{\bm\varphi}_N= (\varphi(\ell h))_{\ell =0,\ldots,N-1}\in\mathbb{C}^N,
\end{aligned}
\]
This matrix $\bm{A}_{\delta,N}$ can be fast computed with
\[
 \bm{A}_{\delta,N} = \bm{A}_{N} \odot\bm{\Delta}_N,
\]
where ``\ $\odot$''  is the Hadamard or element-wise matrix product and $ \bm{A}_{N} := (A(\ell h, m h))_{\ell,m=0}^{N-1}$. Finally,   $\bm{\Delta}_N$ is the matrix given by
\[
\begin{aligned}
\bm{\Delta}_N\ &=\    \frac{1}N  (\bm{P}_N \bm{W}_N)^* \bm{W}_N =\frac{1}N  \bm{W}_N^*\overline{\bm{P}}_N \bm{W}_N  \\
\bm{P}_N\ &=\ {\rm diag}\,(\widehat{\delta}(0),\widehat{\delta}(1),\ldots,\widehat{\delta}(N-1-\mu),\widehat{\delta} (-\mu), \widehat{\delta} (-\mu+1),\dots, \widehat{\delta}(-1)),\\
  \bm{W}_N\ &=\   \left(\exp(-2\pi i \ell m/N)\right)_{\ell,\, m=0}^{N-1}.
 \end{aligned}
\]
That is to say, 
\[
 \bm{\Delta}_N = {\rm iFFT}^\top ( {\rm FFT}(  \overline{\bm{P}}_N  )).
\]
where
${\rm FFT}$ and ${\rm iFFT}^\top$ are the Discrete Fourier Transform and the inverse Discrete Fourier Transform applied  column-wise  and row-wise respectively.
\end{remark}

\subsubsection*{Numerical method}

Using the singular quadratures above, we can define with 
\begin{equation}\label{eq:AcombN}
\mathcal{A}_{{\rm comb},N}^{(2)} =
\begin{bmatrix}
     \D Q_N\V_{p,N}^{(4)} \D +\widetilde{\V}_{p,N }^{(3)}                     &k_s^2   {\V}_{s,\ttt,\nnn,N}
      -\K_{s,N}^\top \D\\
     k_p^2   {\V}_{p,\ttt,\nnn,N} -\K_{p,N}^\top \D\ & -
     \D Q_N\V_{s,N}^{(4)} \D +\widetilde{\V}_{s,N }^{(3)} \\
    \end{bmatrix}-
    \begin{bmatrix}
     \K^\top _{p,N} & \D Q_N\V_{s,N}^{(4)}\\
    \D Q_N\V_{p,N}^{(4)} & - \K^\top_{s,N}
    \end{bmatrix}\begin{bmatrix}
               \mathrm{Y}_p &\\
                      & \mathrm{Y}_{s}
              \end{bmatrix} 
\end{equation}
the discrete approximation of  ${\cal A}_{\rm comb,\mathcal{R}}$ 
\begin{equation}
\begin{aligned}
\mathcal{A}_{{\rm comb},\mathcal{R},N}&:=
\left( \mathcal{A}_{\rm comb,pp}^{(-1)}+  \mathcal{A}_{\rm comb,pp}^{(2)} +
\mathcal{Q}_N
 \mathcal{A}_{{\rm comb},N}^{(2)} \right) { \cal R}   ,\quad
   \mathcal{Q}_N&:=
    \begin{bmatrix}{Q}_N\\
     &{Q}_N
     \end{bmatrix}.
\end{aligned}
\end{equation}

Then, the numerical method can be written as follows: find $( \lambda_{p,N},\, \lambda_{s,N})$ so that
\begin{equation}\label{eq:theMethod}
 \begin{aligned}
\mathcal{A}_{{\rm comb},\mathcal{R},N} \begin{bmatrix}
                                                            \lambda_{p,N}\\
                                                            \lambda_{s,N}
                                                           \end{bmatrix}
=&
\mathcal{Q}_N
\begin{bmatrix}
   f_{\nnn}\\
   f_{\ttt}
   \end{bmatrix}, \\
   &&
\deleted[id=vD]{\mathcal{A}_{{\rm comb},\mathcal{R},N}:=
\left( \mathcal{A}_{\rm comb,pp}^{(-1)}+ \mathcal{A}_{\rm comb,pp}^{(2)}+
\mathcal{Q}_N
 \mathcal{A}_{{\rm comb},N}^{(2)} \right) { \cal R}
   \mathcal{Q}_N:=
    \begin{bmatrix}{Q}_N\\
     &{Q}_N
     \end{bmatrix}.}
\end{aligned}
\end{equation}
Notice that $\lambda_{p,N}, \lambda_{s,N}\in\mathbb{T}_N$ because so does the right hand side. In other words, the true unknowns are the values of $\lambda_{p,N},\ \lambda_{s,N}$ at the grid points $\{mh\}_{m=0}^{N-1}$, $h=2\pi/N$. Besides, this way of writing the equation will simplify the analysis in next subsection. Finally, we can define the densities, the numerical counterpart of \eqref{eq:AR2}, as
\begin{equation}\label{eq:theMethod:02}
 \mathbb{T}_N\times \mathbb{T}_N \ni
\begin{bmatrix*}
\varphi_{p,N}\\
\varphi_{s,N}
\end{bmatrix*}:= {\cal R}\begin{bmatrix}
                                                            \lambda_{p,N}\\
                                                            \lambda_{s,N}
                                                           \end{bmatrix}.
  \end{equation}
%

\subsection{Error analysis}

Our error analysis relies on two main results. The first one is a classical result concerning the error in the trigonometric polynomial interpolation in (periodic) Sobolev norms: for any $q\ge 0$, $r\ge 0$ with $q+r>1/2$, there exists $C_{q,r}>0$ such that
\begin{equation}\label{eq:QN}
 \|Q_N f-f\|_{q}\le C_{q,r}N^{q-r}\|f\|_{q+r}
\end{equation}
for any $N$ and $f\in H^{q+r}$ cf. \cite[Ch. 8]{Saranen}. As a byproduct, we have
\begin{equation}\label{eq:QN:01}
 \|Q_N f\|_{q}\le C_{q}  \|f\|_q
\end{equation}
for $q>1/2$ and, for $q\ge 0$
\begin{equation}\label{eq:QN:02}
\|Q_N f\|_{q}\le  \|f\|_q + C_{q} N^{-1} \|f\|_{q+1}.
\end{equation}
This last inequality will be useful to carry out the error analysis for  Sobolev exponents $q\in[0,1/2)$ since \eqref{eq:QN:01} cannot cover this case.  We will also use the inverse inequalities:
\[
 \|Q_N f\|_{q+r'}\le C N^{r'} \|Q_N f\|_{q}
\]
which hold for any $r'\ge 0$.

Let us stress that we will express the above and forthcoming estimates as operator convergence estimates in appropriate Sobolev norms. Hence, we will write, for instance,
\[
\begin{aligned}
 \|Q_N -\mathrm{I}\|_{H^{q+r}\to H^q}& \le C_{q,r} N^{ -r},\quad &&q\ge 0,\ r\ge 0,\ q+r>1/2,\\
 \|Q_N\mathrm{A}\|_{H^{q+r}\to H^q} & \le C_{q}  \|\mathrm{A}\|_{H^{q+r}\to H^q} + C N^{-1}  \|\mathrm{A}\|_{H^{q+r}\to H^{q+1}}  ,\quad&& q\ge 0,\\
  \|Q_N\|_{H^{q+r}\to H^{q+r'}}& \le   C_{q,r,r'} N^{r'}\|Q_N\|_{H^{q+r}\to H^q}, \quad&&    r'\ge 0,
  \end{aligned}
\]
where $\mathrm{I}$ is the identity and $\mathrm{A}$ is an arbitrary linear continuous operator between the corresponding Sobolev spaces.

The second main result concerns the error in operator norm between the continuous integral operators and the  different semi-discrete approximations proposed in this paper (recall Remark \ref{remark:the:remark}) that can be summarized as follows.

\begin{proposition}\label{prop:main:2a}
Let $\delta\in H^s$, for some $s$, with
\[
|\widehat{\delta}(n)|\le c (1+|n|)^{-m},\quad n\in\mathbb{Z}
\]
and for $A$ smooth and bi-periodic, define
\begin{multline*}
( \mathrm{A}_{\delta} \varphi)(t) :=\frac{1}{2\pi}\int_{0}^{2\pi} A(t,\tau)\delta(t-\tau)\varphi(\tau)\,{\rm d}\tau
= \sum_{n\in\mathbb{Z}} \widehat{\delta}(n)\varphi_A(t;n)  e_n(t)
\\
\varphi_A(t;n) = \widehat{[A(t,\cdot)\varphi]}(n) = \frac{1}{2\pi}\int_0^{2\pi} A(t,\tau)\varphi(\tau)e_{-n}(\tau)\,{\rm d}\tau
\end{multline*}
and its approximation
\begin{multline*}
 ( \mathrm{A}_{\delta,N}\varphi)(t) :=\frac{1}{2\pi}\int_{0}^{2\pi} Q_N(A(t,\cdot)\varphi)(\tau) \delta(t-\tau) \,{\rm d}\tau
 = \sum_{n=-\nu}^{N-\nu-1} \widehat{\delta}(n) \varphi_{A,N}(t;n) e_n(t)\\
  \varphi_{A,N}(t,n)= \widehat{[Q_N(A(t,\cdot)\varphi]}(n)=
  \frac{1}{2\pi}\int_0^{2\pi} Q_N(A(t,\cdot)\varphi)(\tau)e_{-n}(\tau)\,{\rm d}\tau.
\end{multline*}
Then for $q+m>1/2$, with $r\ge -m$  there exists $C_{q,r}>0$ such that
\[
 \| \mathrm{A}_{\delta}   - \mathrm{A}_{\delta,N}  \|_{H^{q+r}\to H^q}
 \le C_{q,r} N^{-r-\min\{q,m\}}.
\]

\end{proposition}
\begin{proof} The proof is based on the same techniques as those presented in \cite[Ch. 12 and 13]{Kress}. We also refer the reader to \cite{BoTuDo:2015} and \cite{DomCat:2015}.
\end{proof}

\begin{remark}\label{remark:convergence:estimates}
Let us point out that the operator norm convergence estimates for the numerical discretizations used in this paper follow in all cases the prototype
 \begin{equation}\label{eq:01:remark}
  \|\mathrm{L}-\mathrm{L}_{N}\|_{H^{q+r}\to H^q}\le C N^{-r-\min\{q,m\}}
 \end{equation}
 where $-m$ is the order of $\mathrm{L}$ (i.e. $\mathrm{A}\in\mathrm{OPS}(-m)$),  with $q+r>1/2$ and $r\ge -m$. Note that the first restriction ensures, by the Sobolev embedding theorem, that these operators act on continuous functions and therefore the interpolating operator $Q_N$, which is essential in the definition of all considered discretizations $\mathrm{L}_N$,  can be applied.

 Hence, \eqref{eq:01:remark}  can be easily verified for the pairs
 $\{\V_k^{(m)},\V^{(m)}_{k,N}\}$ (logarithmic operators with  indices $m\ge 1$),  $\{\mathrm{I},Q_N\}$ ($m=0$) or even for the differentiation operator $\{\D,\D Q_N\}$ (now with $m=-1$).

This observation paves the way for a better understanding of the underlying convergence estimates in the following subsection and greatly facilitates the subsequent proofs of the results that will be presented below.
\end{remark}


\begin{theorem}\label{th:02}
 For any $q\ge 1$, $r\ge 0$ with $q+r> 3/2 $ it holds
 \[
 \left\|\mathcal{A}_{{\rm comb}}^{(2)}\mathcal{R}-
 \mathcal{Q}_N\mathcal{A}_{{\rm comb},N}^{(2)} \mathcal{Q}_N\mathcal{R}
 \right\|_{H^{q+r}\times H^{q+r} \to H^q\times H^q }
 \le C_{q,r}N^{-r-\min\{ 1,q-2\}}
 \]
 with $C_{q,r}$ independent of $N$.
\end{theorem}
\begin{proof}
Note first the order of the operators involved, $\mathcal{A}_{{\rm comb}}^{(2)}$ and $\mathcal{R}$, the estimate of convergence for the interpolating operator and the inverse inequality for $\mathbb{T}_N$ implies that for any $q,r\ge 0$ there exists $C>0$ such that
\begin{eqnarray*}
 \left\|\mathcal{A}_{{\rm comb}}^{(2)}\mathcal{R}-\mathcal{A}_{{\rm comb}}^{(2)}\mathcal{Q}_N\mathcal{R},
 \right\|_{H^{q+r}\times H^{q+r} \to H^q\times H^q }&\le& C  \left\|\mathrm{I}-Q_N
 \right\|_{H^{q+r-1}  \to H^{\max\{q-2,0\}}}\le C N^{ -r-\min\{1, q-1 \}}\\
\|\mathcal{Q}_N\mathcal{R}\|
 _{H^{q+r}\times H^{q+r} \to H^{q+r}\times H^{q+r}}&\le&  C  N,\\
 \left\|(\mathcal{I}-\mathcal{Q}_N)\mathcal{A}_{{\rm comb}}^{(2)}\mathcal{Q}_N\mathcal{R}\|
 \right\|_{H^{q+r}\times H^{q+r} \to H^q\times H^q }&\le& C  \left\|\mathrm{I}-Q_N
 \right\|_{H^{q+r}  \to H^{q }} \|\mathcal{Q}_N\mathcal{R}\|_{H^{q+r}\times H^{q+r} \to H^{q+r}\times H^{q+r} },\\
 &\le& C N^{-1-r}
 \end{eqnarray*}
(second bound is just an inverse inequality)
and therefore
 \begin{multline*}
 \left\|\mathcal{A}_{{\rm comb}}^{(2)}\mathcal{R}-
 \mathcal{Q}_N\mathcal{A}_{{\rm comb},N}^{(2)} \mathcal{Q}_N\mathcal{R}
 \right\|_{H^{q+r}\times H^{q+r} \to H^q\times H^q }\\
  \le C
  N \left\|\mathcal{A}_{{\rm comb}}^{(2)}-
\mathcal{A}_{{\rm comb},N}^{(2)}
 \right\|_{H^{q+r}\times H^{q+r} \to H^q\times H^q } +  C N^{-r-\min\{ 1,q-1\} }.
 \end{multline*}
According to the definition of the operators involved (see \eqref{eq:Acomb} and \eqref{eq:AcombN}), and noticing that $\D, \mathrm{Y}_k\in\mathrm{OPS}(1)$,  we see that the results is consequence of
\begin{subequations}\label{eq:02:proof:main:Th}

\begin{eqnarray}
 \|\D\V_{k}^{(4)}-Q_N\D\V_{k,N}^{(4)}\|_{H^{q+r-1}\to H^{q}}\\
 &&\hspace{-3cm}\ \le\ \|\D\V_{k}^{(4)}-{\mathrm Q}_N\D\V_{k}^{(4)}\|_{H^{q+r-1}\to H^{q}}+\|{\mathrm Q}_N\D\V_{k}^{(4)}-{\mathrm Q}_N\D\V_{k,N}^{(4)}\|_{H^{q+r-1}\to H^{q}}\nonumber\\
 &&\hspace{-3cm}\ \le\
 C_{q,r} N^{-r-2} +\|{\mathrm Q}_N\D\|_{H^{q+1}\to H^{q}}
 \|\V_{k}^{(4)}-\V_{k,N}^{(4)}\|_{H^{q+r-1}\to H^{q+1}}\nonumber\\
 &&\hspace{-3cm}\ \le\ C_{q,r} N^{-r-\min\{2,q-1\}}
  \label{eq:02a:proof:main:Th}
\end{eqnarray}
and the estimates (recall that ${ \V}_{k,\ttt,\nnn}$ is a logarithmic operator of order $-2$)
\begin{eqnarray}
 \| { \V}_{k,\ttt,\nnn}-{\V}_{k, \ttt,\nnn,N}\|_{H^{q+r}\to H^{q}}&\le&  C_{q,r}N^{ -r-\min\{2,q\}} \label{eq:02b:proof:main:Th},\\
 \| \widetilde{\V}_{k}^{(3)}-\widetilde{\V}_{k,N}^{(3)}\|_{H^{q+r }\to H^{q}}
 &\le&
 C_{q,r}N^{ -r-\min\{3,q\}} \label{eq:02c:proof:main:Th},\\
 \| { \K}_{k}^\top - {\K}_{k,N}^\top \|_{H^{q+r-1}\to H^{q}}&\le&  C_{q,r}N^{  -r-\min\{2,q-1\}}.  \label{eq:02d:proof:main:Th}
 \end{eqnarray}

\end{subequations}
The result is then proven.

\end{proof}

 We are ready to state the stability and convergence of the method:

\begin{theorem}\label{theo:4.5}
For $N$ large enough  the equations of the numerical method \eqref{eq:theMethod} admits a unique solution $(\lambda_{p,N},\lambda_{s,N})$ which satisfies, for any $q>2$,
\[
\| \lambda_{p,N}\|_{q}+\|\lambda_{s,N}\|_{q}\le C_q \big(
\|  \lambda_{p}\|_{q}+\|\lambda_{s}\|_{q}\big)
\]
with $C$ independent of $(\lambda_p,\ \lambda_s)$, the exact solution of \eqref{eq:AR}, and $N$.
Furthermore, we have the following error estimate
\[
\| \lambda_{p}-\lambda_{p,N}\|_{q}+\|\lambda_{s}-\lambda_{s,N}\|_{q}\le  C_{q,r}  N^{-r}\big(
 \|  \lambda_{p}\|_{q+r}+\|\lambda_{s}\|_{q+r}\big)
\]
with  $C_{q,r}>0$ independent of $N$ and $(\lambda_p,\lambda_s)$.
\end{theorem}
\begin{proof} We recall again the relation between exact
\begin{equation}\label{eq:Proof:00}
\mathcal{A}_{\rm comb}
{ \cal R}\begin{bmatrix}
                                                            \lambda_{p}\\
                                                            \lambda_{s}
                                                           \end{bmatrix} =
\left(   \mathcal{A}_{\rm comb,pp}^{(-1)} + \mathcal{A}_{\rm comb,pp}^{(2)} +
 \mathcal{A}_{{\rm comb}}^{(2)}\right){ \cal R}\begin{bmatrix}
                                                            \lambda_{p,N}\\
                                                            \lambda_{s,N}
                                                           \end{bmatrix}
=
\begin{bmatrix}
    f_{\nnn}\\
     f_{\ttt}
   \end{bmatrix}
\end{equation}
and numerical (see \eqref{eq:theMethod}) solution:
\begin{equation}\label{eq:Proof:01}
\mathcal{A}_{{\rm comb},N}
{ \cal R}\begin{bmatrix}
             \lambda_{p,N}\\
            \lambda_{s,N}
           \end{bmatrix} =
\left(  \mathcal{A}_{\rm comb,pp}^{(-1)}+ \mathcal{A}_{\rm comb,pp}^{(2)} +
\mathcal{Q}_N
 \mathcal{A}_{{\rm comb},N}^{(2)}\mathcal{Q}_N \right)  { \cal R}  \begin{bmatrix}
    \lambda_{p,N}\\
    \lambda_{s,N}
               \end{bmatrix}
=
\begin{bmatrix}
   Q_N f_{\nnn}\\
   Q_N f_{\ttt}
   \end{bmatrix}.
\end{equation}
As consequence of Theorem  \ref{th:02}, and for $q>2$,
\[
\left\| (\mathcal{A}_{{\rm comb}, N}^{(2)}- \mathcal{A}_{{\rm comb}}^{(2)})\mathcal{R}\right\|_{H^q\times H^q\to H^q\times H^q}\to 0
\]
and therefore $\mathcal{A}_{{\rm comb}, N}\mathcal{R}:H^q\times H^q\to H^q\times H^q$
is uniformly continuous with uniformly continuous inverse  provided that $N$ is large enough.

Furthermore,
\begin{eqnarray*}
\| \lambda_{p,N} -\lambda_p\|_{q}+
\| \lambda_{s,N} -\lambda_s\|_{q} &\le&  \left\|
\mathcal{A}_{{\rm comb}, N} \mathcal{R}\begin{bmatrix}
 \lambda_{p,N} -\lambda_p\\
  \lambda_{s,N} -\lambda_s
\end{bmatrix}
\right\|_q \\
&& \hspace{-3cm} = \  \left\|\left( \mathcal{A}_{{\rm comb}}^{(2)}{ \cal R}  -\mathcal{Q}_N
 \mathcal{A}_{{\rm comb},N}^{(2)}   \mathcal{Q}_N{ \cal R}\right)  \begin{bmatrix}
  \lambda_p\\
  \lambda_q
  \end{bmatrix}\right\|_{q} \\
&& \hspace{-3cm}\qquad   +
   \left\|
  \begin{bmatrix}
{Q}_N f_{\nnn}-f_{\nnn} \\
    {Q}_N f_{\ttt}-f_{\ttt}
    \end{bmatrix}
    \right\|_{q}\\
&& \hspace{-3cm} \le \  C_{q,r} N^{-r-\min\{ 1,q-2\}} \big(
 \|  \lambda_{p}\|_{q+r}+\|\lambda_{s}\|_{q+r}\big)+
   C'_{q,r} N^{-r} \big(
 \|f_{\nnn}\|_{q+r}+\|f_{\ttt}\|_{q+r}\big)\\
&& \hspace{-3cm} \le \   C''_{q,r}N^{-r}  \big(
 \|  \lambda_{p}\|_{q+r}+\|\lambda_{s}\|_{q+r}\big)
\end{eqnarray*}
where in the last step we have used the continuity $\mathcal{A}_{{\rm comb}}\mathcal{R}: H^{q+r}\times H^{q+r}\to H^{q+r}\times H^{q+r}$.
\end{proof}

\begin{corollary}
Let $(\varphi_p,\varphi_s)$ be the given  by \eqref{eq:AR2}, and $(\varphi_{p,N},\varphi_{s,N})$ that given by \eqref{eq:theMethod:02}. Then for any $N$ large enough, and for any $q>1$, $r\ge 0$ there exists $C_{r,q},\ C'_{r,q}>0$ so that
\[
\|   \varphi_{p}-\varphi_{p,N}\|_{q}+\|\varphi_{s}-\varphi_{s,N}\|_{q}\le  C_{q,r}  N^{-r}\big(
 \|  \varphi_{p}\|_{q+r}+\|\varphi_{s}\|_{q+r}\big)\le C'_{q,r}  N^{-r}\big(
 \|  f_{\nnn}\|_{q+r+1}+\|f_{\ttt}\|_{q+r+1}\big)
\]
\end{corollary}
\begin{proof}
It is straightforward, since
\[
 \begin{bmatrix}
  \varphi_{p}-\varphi_{p,N}\\
  \varphi_{s}-\varphi_{s,N}
 \end{bmatrix}
= {\cal R} \begin{bmatrix}
  \lambda_{p}-\lambda_{p,N}\\
  \lambda_{s}-\lambda_{s,N}
 \end{bmatrix}
\]
${\cal R}\in\mathrm{OPS}(1)$ and, by Theorem \ref{theo:01}, ${\cal A}_{\rm comb}{\cal R}\in{\rm OPS}(0)$ is invertible.
\end{proof}

\section{Arbitrary parametrizations}\label{gen}

Although arc-length parametrizations ${\bf x}:[0,2\pi]\to\Gamma$
can be constructed for rather general geometries, it is important to analyze Nystr\"om discretizations for arbitrary smooth $2\pi$-periodic parametrizations
$\widecheck{\bf x}:\mathbb{R}\to\Gamma$,
and we show in this section how to extend the analysis to the general case. The main difficulty arises from the fact $|{\bf x}'|$ is no longer constant, which in turn makes  the regularizing operators $\mathcal{Y}$ and $\mathcal{R}$, as well as the tangential derivative $\mathrm{\partial}_{\ttt}$, to no longer be Fourier multiplier operators. Furthermore, the principal part operator $\mathcal{A}_{\rm comb,pp}^{(-1)}+ \mathcal{A}_{\rm comb,pp}^{(2)} $ requires a more delicate analysis in order to extract its principal symbol in Fourier multiplier form, so that the remainder is a sufficiently regular operator, a trick which is essential to the numerical analysis we have presented in the previous sections.

\subsection{The construction of the regularizing operators}
Let us then start then from an arbitrary smooth $2\pi-$periodic parametrization $\widecheck{\bf x}:\R\to{\Gamma}$  and set from now on
\begin{equation}\label{eq:eta}
\eta(\tau)= |\widecheck{\bf x}'(\tau)|
\end{equation}
the norm of the parametrization. Then
\[
 \widecheck{\D}\varphi := \frac{1}{\eta}\D\varphi = (\partial_{\bm t}\varphi_\Gamma)\circ\widecheck{\bf x}
\]
is just the parameterized tangent derivative. Define also
\[
  \widecheck {\D}_{-1}\varphi =  \D_{-1}(\eta \varphi)
\]
so that for any  $\varphi$ of zero mean.
\[
 \widecheck{\D}_{-1}  \widecheck {\D}\varphi    =\varphi,\quad \text{if }\int_0^{2\pi}
\varphi \eta = \int_{\Gamma}\varphi_\Gamma = 0.
\]
We extend this operator for negative integer values of $r$,
\[
   \widecheck {\D}_{-r}\varphi  =  \D_{-r} (\eta^r\varphi)
\]
and introduce the averaging operator accordingly:
\[
\widecheck{\J}\varphi:= \J(\eta\varphi) = \frac{1}{2\pi}\int_{\Gamma}\varphi_\Gamma.
\]
Notice then $ \widecheck{\J} 1 =1$ (recall Remark \ref{remark:curves:2pi}).

On the other hand, it is not difficult to prove that  (recall the notation $\varphi  = \varphi_\Gamma \circ\widecheck {\bf x}$ and
{assume the parameterizations ${\bf x}$ and $\widecheck{\bf x}$ both start from the same initial points})
\[
 \HH\varphi - \mathrm{p.v.}\int_{0}^{2\pi}  \frac{\eta(\tau)}{|{\bf x}(\cdot)-{\bf x}(\tau)|}\varphi(\tau)\,{\rm d}\tau =
  \HH\varphi - \mathrm{p.v.}\int_{\Gamma}  \frac{1}{|  {\bf x}(\cdot)-{\bf y}|}\varphi_\Gamma ({\bf y})\,{\rm d}{\bf y}
  \in \mathrm{OPS}(-\infty).
\]
(It suffices to show that, using \eqref{eq:def:H},  the difference between the two operators is an integral operator with smooth kernel)
As a simple consequence, if  for any parametrization $\bm{z}$ we set $\HH_{\bf z}:H^s(\Gamma)\to H^s(\Gamma)$  the operator defined by
\[
(\HH_{\bf z}\varphi_\Gamma){\circ {\bf z}} :=\HH  \varphi,
\]
we have that   $\HH_{{\bf x}}-\HH_{\widecheck{\bf x}}$ is a smoothing operator.  In other words, the  Hilbert transform when seeing acting on functions on $\Gamma$ by means of two different smooth parametrizations differ by an operator of order $-\infty$.

The (parameterized) Dirichlet-To-Neumann
and the regularizer operator becomes now
\begin{equation} \label{eq:Ykcheck}
 \widecheck{\mathrm{Y}}_k=  \eta^{-1} \D\HH +\frac{\widetilde{k}^2}2 \HH \D _{-1}\eta
 +\eta^{-1} \J,\quad
  \widecheck{\mathcal{Y}} := \begin{bmatrix}
                                    \widecheck{\mathrm{Y}}_p&\\
                                   &\widecheck{\mathrm{Y}}_s
                                  \end{bmatrix}\in\mathrm{OPS}(1).
\end{equation}
Note that $\eta^{-1} \J$ is used above for $\widecheck{\mathrm{Y}}_k$ instead of the perhaps more natural $\widecheck{\J}=\J\eta$ because it simplifies the analysis. Note, however, that these two alternatives differ in a $\mathrm{OPS}(-\infty)$ operator.

We also define
\begin{equation}
\begin{aligned}
\widecheck{\mathcal{R}}&\ =\
\frac{1}{\eta}
 \mathcal{H}_0
 \HH\D
+\frac{1}{2}\begin{bmatrix}
             \widetilde{k}_s^2&\\
                            & -\widetilde{k}_p^2
            \end{bmatrix}\HH\widecheck{\D}_{-1}
            + \begin{bmatrix}
              \widecheck{\J}&\\
                            &\widecheck{\J}
            \end{bmatrix} \\
& \ =\
 \frac{1}{\eta}
 \mathcal{H}_0
 \HH\D
+\frac{1}{2}\begin{bmatrix}
             \widetilde{k}_s^2&\\
                            & -\widetilde{k}_p^2
            \end{bmatrix}\HH\D_{-1}\eta
            + \begin{bmatrix}
              {\J}&\\
                            & {\J}
            \end{bmatrix}\eta\in\mathrm{OPS}(1),\qquad
 \mathcal{H}_0=\begin{bmatrix}
 \mathrm{I}& -\mathrm{H} \\
 -\mathrm{H}&-\mathrm{I}
 \end{bmatrix}.
 \end{aligned}
\end{equation}
The newly defined operators $\widecheck{\mathrm{Y}}_k$ and respectively $\widecheck{\mathcal{R}}$ are parametrization invariant regularizers  (module regularizing operators of order $-2$). Indeed, it is a simple matter to prove the following regularizing properties with respect to arclength parametrization versions
\[
\begin{aligned}
 \widecheck{\mathrm{Y}}_k \varphi-
 (\mathop{\rm DtoN}\nolimits_k \varphi_\Gamma)\circ{\bf x}
&= \mathrm{OPS}(-2),\quad
 \widecheck{\mathcal{R}}(\varphi_p,\varphi_s)-(\mathcal{R}(\varphi_{p,\Gamma},\varphi_{s,\Gamma}))  \circ{\bf x} &= \mathrm{OPS}(-2).
 \end{aligned}
\]
The first result above follows from \eqref{eq:Yk-2Wk} and \eqref{exp:Wkhat} below. The second result is obtained by substituting $\D_{m}$ with $\widecheck{\D}_{m}$ and utilizing the fact that $\eta \HH-\HH\eta\in{\rm OPS}(-\infty)$ (see discussion about \eqref{eq:def:Cinfty} below), which leads to the simpler expression for $\mathcal{R}$ as proposed above. Furthermore, the following result holds

\begin{proposition}
We have
\begin{enumerate}[label =(\alph*)]
 \item\label{prop:1}
$\widecheck{\mathcal{Y}}_k \in \mathrm{OPS}(1)$ and invertible.
\item\label{prop:2}
$\widecheck{\mathcal{R}} \in\mathrm{OPS}(1)$ and injective.
\end{enumerate}
\end{proposition}
\begin{proof}
We will prove only \ref{prop:2}, since \ref{prop:1} can be shown  using the same techniques. Let $(\varphi_1,\varphi_2)\in \mathcal{N}(\widecheck{\mathcal{R}})$ and consider
 \begin{eqnarray*}
  \big\langle \begin{bmatrix}\eta\overline{\varphi_1}\\-\eta\overline{\varphi_2}
          \end{bmatrix}
,\widecheck{\mathcal{R}} \begin{bmatrix} { \varphi_1}\\  { \varphi_2}
           \
          \end{bmatrix}\big\rangle
          &:=& (\overline{\varphi}_1,\HH\D\varphi_1)+ (\overline{\varphi}_2,\HH\D\varphi_2)
         + (\overline{\varphi}_1,\D\varphi_2)- (\overline{\varphi}_2,\D\varphi_1)\\
         && +\frac12\widetilde{k}_p^2(\overline{\eta\varphi}_1,\HH\D_{-1}\eta\varphi_1)+
         \frac12\widetilde{k}_s^2(\overline{\eta\varphi}_2,\HH\D_{-1}\eta\varphi_2)+
          |\J\eta\varphi_1|^2- |\J\eta\varphi_2|^2.
 \end{eqnarray*}
Using integration by parts and noticing that $\HH\D $ is positive semi-definite we have that the imaginary part of the scalar product above is given by
 \[
  \Im  \langle \begin{bmatrix}\eta\overline{\varphi_1}\\-\eta\overline{\varphi_2}
          \end{bmatrix}
,\widecheck{\mathcal{R}} \begin{bmatrix} { \varphi_1}\\  { \varphi_2}
           \
          \end{bmatrix}\rangle = \frac12\Im \left( \widetilde{k}_p^2(\overline{\eta\varphi}_1,\HH\D_{-1}\eta\varphi_1)+
         \widetilde{k}_s^2(\overline{\eta\varphi}_2,\HH\D_{-1}\eta\varphi_2)\right).
 \]
 But,
\[
  \Im \left( \widetilde{k}_p^2(\overline{\eta\varphi}_1,\HH\D_{-1}\eta\varphi_1)+
         \widetilde{k}_s^2(\overline{\eta\varphi}_2,\HH\D_{-1}\eta\varphi_2)\right) =0
\]
if and only if $\eta\varphi_1$ and $\eta\varphi_2$ are constants (notice that $\Im \widetilde{k}_p^2  ,\  \Im \widetilde{k}_s^2 > 0$ and that $\HH\D_{-1}$ is positive definite in $H_0^{1/2}:=\{\varphi\in H^{1/2}\ :\ \widehat{\varphi}(0)=0\}$).

Assume then that $(\varphi_1,\varphi_2)=(\alpha_1\eta^{-1},\alpha_2\eta^{-1})\in \mathcal{N}(\widecheck{\mathcal{R}})$. Then,
\[
 {\bf 0}=\widecheck{\mathrm{J}}\widecheck{\mathcal{R}}\begin{bmatrix}
                            \varphi_1\\
                            \varphi_2
                           \end{bmatrix}
=\mathrm{J}\eta \left(\mathrm{J}\eta\begin{bmatrix}
                            \varphi_1\\
                            \varphi_2
                           \end{bmatrix}\right)
                           =\begin{bmatrix}
                            \alpha_1\\
                            \alpha_2
                           \end{bmatrix}
\]
from which we conclude $\alpha_1=\alpha_2=0$.

\end{proof}

Having presented the construction of the regularizer operators, we are ready to introduce the parametrized combined field regularized formulation. Hence,  set
 \begin{subequations}\label{eq:theequations:hat}
  \begin{eqnarray}\label{eq:firstEq:hat}
 \widecheck{\mathcal{A}}_{\rm comb}   &=&
 \widecheck{\mathcal{A}}_{\rm DL}-    \widecheck{\mathcal{A}}_{\rm SL} \widecheck{\cal Y} ,\quad \widecheck{\mathcal{Y}}  = \begin{bmatrix}
                                              \widecheck{\mathrm{Y}}_{p} & \\
                                            & \widecheck{\mathrm{Y}}_{s}
                                       \end{bmatrix} \\
 \widecheck{\mathcal{A}}_{\rm DL}&:=&
\begin{bmatrix}
\widecheck{\W}_{p}&\frac{1}2 \widecheck{\D} + k_{s}^2\ttt \cdot \widecheck{\V}_{s}[\nnn\,\cdot\,]-\widecheck{\K}_{s}^\top \widecheck{\D}       \\
\frac{1}2 \widecheck{\D} + k_{p}^2\ttt \cdot \widecheck{\V}_{p}[\nnn\,\cdot\,]-\widecheck{\K}_{p}^\top \widecheck{\D}  &-\widecheck{\W}_{s}
\end{bmatrix}\label{eq:ADL:check}
\\
\widecheck{\mathcal{A}}_{\rm SL} &:=&
\begin{bmatrix}
-\frac{1}2 \I + \widecheck{\K}_{p}^\top &  \widecheck{\D}\widecheck{\V}_{s}  \\
\widecheck{\D} \widecheck{\V}_{p}  &\frac{1}2 \I - \widecheck{\K}_{s}^\top
\end{bmatrix}.\label{eq:ASL:check}
\end{eqnarray}
\end{subequations}
Note that the following notation convention has been employed:
\[
  \widecheck{\mathrm{A}}_k  \widecheck{\varphi}  = (\mathrm{A}_{k,\Gamma}  \varphi_\Gamma)\circ \widecheck{\bf x}      , \quad  \widecheck{\mathrm{A}}_k  \in\{\V_k,\K_k^\top, \W_k\}
\]
(As a reminder, in this work we have used the convention: $\widecheck{\varphi}=\varphi_\Gamma\circ\widecheck{\bf x}$).

Then, the boundary integral formulation, counterpart  of \eqref{eq:AR}-\eqref{eq:AR2},  is the following: For
\[
    \widecheck{f}_{\nnn} :=
   (\gamma_{\Gamma} {\bf u}^{\rm inc} \cdot {\bm n})\circ\widecheck{\bf x},\quad
   \widecheck{f}_{\ttt}:=
   (\gamma_{\Gamma} {\bf u}^{\rm inc} \cdot {\bm t})\circ\widecheck{\bf x}
\]
we solve first
 \begin{equation}\label{eq:ARcheck}
  \widecheck{\mathcal{A}}_{\rm comb} \widecheck{\mathcal{R}}\begin{bmatrix}
                              \widecheck\lambda_{p}\\
                              \widecheck\lambda_{s}
                             \end{bmatrix}
={ -}\begin{bmatrix}
   \widecheck{f}_{\nnn}  \\
   \widecheck{f}_{\ttt}
   \end{bmatrix},\quad
 \end{equation}
 followed by
 \begin{equation}\label{eq:ARcheck2}
\begin{bmatrix}
                              \widecheck\varphi_{p}\\
                              \widecheck\varphi_{s}
                             \end{bmatrix}   =
\widecheck{\mathcal{R}}\begin{bmatrix}
                              \widecheck\lambda_{p}\\
                              \widecheck\lambda_{s}
                             \end{bmatrix}.
 \end{equation}
The Helmholtz decomposition for the solution of the Navier equation  ${\bm u}$ can be next constructed from the pair $(\widecheck\varphi_{p},\
\widecheck\varphi_{s})$ using the boundary layer potential ansatz \eqref{eq:up_and_us} with $\widecheck{\mathrm{Y}}_k$ instead.

\subsection{ Parameterized Helmholtz BIO}

We start from the single layer operator for which it is possible to prove
 \begin{eqnarray*}
 \widecheck{ \mathrm{V}}_k &=&
\frac{1}2\HH\widecheck{\D}_{-1}  +\frac{k^2}{4}\HH\widecheck{\D}_{-3}+
  \widecheck{ \mathrm{V}}^{(4)}_k
 \end{eqnarray*}
where
\begin{equation}\label{eq:V5}
 \widecheck{ \mathrm{V}}^{(4)}_k=
\frac{k^2}4 (\Lambda_3- \HH\D_{-3}) \eta^3 +
 \widetilde{ \mathrm{V}}^{(4)}_k\in\mathrm{OPS}(-4)
\end{equation}
and
 \[
   \widetilde{ \mathrm{V}}^{(4)}_k \varphi = \int_0^{2\pi}  \widecheck{A}^{(4)}(\cdot,\tau)(e_1(\cdot-\tau)-1)^3\log\left(4\sin^2 \frac{\cdot-\tau}2\right)\varphi(\tau)\,{\rm d}\tau +
    \int_0^{2\pi}  \widecheck{B}^{(4)}(\cdot,\tau) \varphi(\tau)\,{\rm d}\tau
 \]
 for certain smooth bi-periodic functions $\widecheck{A}^{(4)}$ and $\widecheck{B}^{(4)}$. The proof follows from the same techniques used in Lemma \ref{lemma:2.2} but noticing now that $\eta(\tau)=|\widecheck{\bf x}'(\tau)|\ne 1$,

Next, it holds
\begin{eqnarray*}
 \widecheck{\W}_k &=&  \widecheck{\D} \widecheck{\V}_k\eta \widecheck{\D} + k^2  \ttt\cdot \widecheck{\V}_{k}( \ttt\cdot).
\end{eqnarray*}
Then, straightforward calculations show
 \begin{eqnarray*}
 \widecheck{\W}_k  &=& \frac1{2\eta} {\HH}\D - \frac{k^2}{4}\eta^{-1}{\HH}{\D}_{-2}\eta^2\D +
 \widecheck{\D}\widecheck{\V}_k^{(4)} {\D} + \frac12{k^2}\HH\D_{-1}\eta + k^2 \widecheck{\V}^{(2)}_{k, \ttt}
 \end{eqnarray*}
with
 \[
  \widecheck{\V}^{(2)}_{k, \ttt}\varphi = \int_{0}^{2\pi}  \widecheck{A}_{\ttt}(
  \replaced[id = vD]{\cdot}{t},\tau)(e_{1}(\replaced[id = vD]{\cdot}{t}-\tau)-1)\log\left(4 \sin^2\frac{\replaced[id = vD]{\cdot}{t}-\tau}2\right)\varphi(\tau)\,{\rm d}\tau+
  \int_{0}^{2\pi}  \widecheck{B}_{\ttt}\added[id = vD]{(\cdot,\tau)} \varphi(\tau)\,{\rm d}\tau\in\mathrm{OPS}(-2)
 \]
where
\[
\begin{aligned}
 \widecheck{A}_{\ttt}(t,\tau) &= \frac{1}{(e_{1}(t-\tau)-1)}\left[\widecheck{A}^{(1)}(t,\tau) (\ttt(t)\cdot\ttt(\tau)) +\frac{1}{4\pi} \right], \\
\widecheck{B}_{\ttt}(t,\tau) &= \widecheck{B}^{(1)}(t,\tau) (\ttt(t)\cdot\ttt(\tau)),
\end{aligned}
\]
$\widecheck{A}^{(1)}$, $\widecheck{B}^{(1)}$ being the functions arising in the splitting in logarithmic and smooth part of the kernel of $\widecheck{\V}_k$ .

Using
 \[
  \eta^2\D\eta^{-1} = \D \eta -2\eta'\mathrm{I}.
 \]
 we can finally rewrite
 \begin{eqnarray*}
 \widecheck{\W}_k
 &=&\frac1{2\eta} {\HH} {\D}+
  \frac{k^2}{4} {\HH}{\D}_{-1}\eta +
  \frac{k^2}4 ( (\HH\D_{-1}-\eta^{-1}\HH\D_{-1}\eta) +2\eta^{-1}\HH\D_{-2}\eta')\eta
\\
&&
+\eta^{-1}{\D}\widecheck{\V}_k^{(4)}\eta^{-1}{\D} +k^2 \widecheck{\V}^{(2)}_{k, \ttt}.
 \end{eqnarray*}

Hence,
\begin{equation}\label{exp:Wkhat}
  \widecheck{\W}_k  =  \frac1{2\eta} {\HH} {\D} +
  \frac{k^2}{4} {\HH}{\D}_{-1} \eta + \widecheck{\W}_k^{(2)}
\end{equation}
with
\begin{equation}\label{eq:Wkcheck:02}
\begin{aligned}
\widecheck{\W}_k^{(2)}\ := \ &
\frac{k^2}4
\mathrm{C}^{(2)}_{\eta}\eta
+ \frac{k^2}{4\eta} \HH\D_{-2}(\eta^2)'
  +
\eta^{-1} \D\widecheck{\V}_k^{(4)}\D +k^2\widecheck{\V}^{(2)}_{k, \ttt}
 \in
 \mathrm{OPS}(-2).
 \end{aligned}
\end{equation}
We have defined above
\begin{subequations}\label{eq:Ca}
\begin{equation}\label{eq:Ca:01}
\begin{aligned}
(\mathrm{C}_a^{(2)} \varphi)(t) &\ :=\  (\HH\D_{-1}\varphi)(t)-a^{-1}(t)(\HH\D_{-1}a\varphi)(t) \\
&\ =\
 -\frac{1}{4\pi}\int_0^{2\pi} \frac{a(t)-a(\tau)}{a(t)(e_{1}(t-\tau)-1)}(e_{1}(t-\tau)-1)\log\left(4 \sin^2\frac{t-\tau}2\right)\varphi(\tau)\,{\rm d}\tau
 \\
 &\ =\
\frac{1}{2\pi} \int_0^{2\pi} r_a(t,\tau)\rho_1(t-\tau)\varphi(\tau)\,{\rm d}\tau\in\mathrm{OPS}(-2)
\end{aligned}
\end{equation}
where
\begin{equation}\label{eq:Ca:2}
r_a(t,\tau):= \begin{cases}
-ia'(t) & t-\tau= 2\pi\ell,\ \ell\in\mathbb{Z}\\
\displaystyle\frac{a(t)-a(\tau)}{a(t)(e_{1}(t-\tau)-1)}, & \text{otherwise}\\
              \end{cases}
\end{equation}
\end{subequations}
turns out to be a smooth bi-periodic function. Notice then, equivalently,
\[
 (\mathrm{C}_a^{(2)} \varphi)(t) = \sum_{n =-\infty}^\infty \left[\widehat{\rho}_0(n)-\widehat{\rho}_0(n-1)\right] \varphi_{r_a}(t;n) e_n(t)
 = \sum_{n =-\infty}^\infty  \widehat{\rho}_1(n)  \varphi_{r_a}(t;n) e_n(t)
\]
where
\[
 \varphi_{r_a}(t;n) = \widehat{[r_a(t,\cdot)\varphi]}(n).
\]
The analysis is quite similar for the operators $\widecheck{\K}^\top_{k}$, and $ \widecheck{\V}_{k,\ttt,\nnn}:=\ttt \cdot \widecheck{\V}_{k}[\nnn\,\cdot\,]$, and therefore we omit further details  for these operators for the sake of brevity.

Summarizing, we have
 \begin{equation}\label{eq:firstEq:v2}
 \begin{aligned}
 \widecheck{\mathcal{A}}_{\rm comb} &\  :=\
 {\widecheck{\mathcal{A}}_{\rm DL,pp}^{}} - {\widecheck{\mathcal{A}}_{\rm SL,pp}^{}}\eta  \widecheck{\mathcal{Y}}+{\widecheck{\mathcal{A}}_{\rm comb}^{(2)}}.
 \end{aligned}
 \end{equation}
with
\begin{equation}\label{eq:Acomb:02}\begin{aligned}
 {\widecheck{\mathcal{A}}_{\rm DL,pp}^{}}  \ &:=\  \frac{1} {2\eta} \mathcal{H}_0  \HH\D  +\frac{1}4  \begin{bmatrix}
                                         k_{p}^2 & \\
                                            & -k_{s}^2
                                       \end{bmatrix}\HH {\D}_{-1}\eta\in\mathrm{OPS}(1) \\
{\widecheck{\mathcal{A}}_{\rm SL,pp}^{}}  \ &:=\
    \frac{1}{2\eta}\mathcal{H}_0
    +\frac{1}{4\eta}\begin{bmatrix}
                                                    &k_{s}^2 \\
                                              k_{p}^2 &
                                       \end{bmatrix}\HH {\D}_{-2}\eta^2
                                        +\frac{1}{2\eta}\begin{bmatrix}
                   0&i\\
                   i&0
                  \end{bmatrix} \J
  \in\mathrm{OPS}(0)
  \\
{\widecheck{\mathcal{A}}_{\rm comb}^{(2)}}  \ &:=    \begin{bmatrix}
     \widecheck{\W}^{(2)}_{p}                      &k_s^2 \widecheck{\V}_{s,\ttt,\nnn}-\widecheck{\K}_s^\top \widecheck{\D} \\
     k_p^2 \widecheck{\V}_{p,\ttt,\nnn}-\widecheck{\K}^\top_p\widecheck{\D}  & -\widecheck{\W}_{s}^{(2)}\\
    \end{bmatrix}  -
    \begin{bmatrix}
     \widecheck{\K}^\top _{p} & \widecheck{\D} \widecheck{\V}_{s}^{(4)}\\
    \widecheck{\D} \widecheck{\V}_{p}^{(4)} & - \widecheck{\K}^\top_{s}
    \end{bmatrix}  \widecheck{\mathcal{Y}}\in\mathrm{OPS}(-2)
\end{aligned}
\end{equation}
where, according to~\eqref{eq:Ykcheck},
\[
   \widecheck{\mathcal{Y}}  :=  \frac{1}{\eta}\begin{bmatrix*}
                                    \I&\\
                                      &\I
                                   \end{bmatrix*}
 (\D\HH+\J) +\frac12
 \begin{bmatrix*}
                                    \widetilde{k}_p^2\\
                                      &\widetilde{k}_s^2
                                   \end{bmatrix*}\HH \D _{-1}\eta.
\]
We notice, however, that unlike what happens with arc-length parametrizations, operators $ {\widecheck{\mathcal{A}}_{\rm DL,pp}^{}}$,  ${\widecheck{\mathcal{A}}_{\rm SL,pp}^{}}$, $\widehat{\mathcal{Y}}$, and neither $\widehat{\mathcal{R}}$, are  Fourier multipliers. Further manipulations are needed to decompose them as sums of Fourier multipliers and sufficiently smoothing operators.

\subsection{ Nystr\"om discretization}

We are now in the position to describe the semi-discretizations of some of the integral operators entering the regularized formulations.   First, in order to deal with ${\widecheck{\mathcal{A}}_{\rm comb}^{(2)}}$, we set
\begin{subequations}
 \begin{align}
 \widecheck{\V}_{k,N}^{(4)}\ := \ &
\frac{k^2}4 (\Lambda_3- \HH\D_{-3}) Q_N \eta^3  +\widetilde{\V}_{k,N}^{(4)}\approx
 \widecheck{\V}_{k}^{(4)}\label{eq:5.10a}
\\
 \widecheck{\W}_{k,N}^{(2)}\ := \ &
 \frac{k^2}4 \mathrm{C}_{\eta,N}^{(2)}\eta
 +\frac{k^2}4 Q_N \eta^{-1}\HH\D_{-2}Q_N (\eta^2)'\label{eq:5.10b}
  +
  \eta^{-1}{\D}Q_N\V_{k,N}^{(4)}\eta^{-1}{\D} +\widecheck{\V}^{(2)}_{k, \ttt,N}\approx
   \widecheck{\W}_{k}^{(2)} \\
  \mathrm{C}_{a,N}^{(2)}\varphi \added[id = vD]{(t)}  \ := \ &
\frac{1}{2\pi} \int_0^{2\pi} Q_N( r_a(t,\cdot)\varphi)(\tau)\rho_1(t-\tau)\,{\rm d}\tau
=\sum_{n=-\nu}^{N-\nu-1} \widehat{\rho}(n) \varphi_{r_a,N}(t;n) e_n(t)\label{eq:5.10c}
\end{align}
\end{subequations}
(see \eqref{eq:Ca}  for the last expression) as well as $\widecheck{\K}^\top_{k,N}\approx \widecheck{\K}^\top_{k} $, and $
\widecheck{\V}_{k,\ttt,\nnn,N} \approx \widecheck{\V}_{k,\ttt,\nnn}$ .

We also introduce the discretization for $\widecheck{\mathcal{R}}$ and $\widecheck{\mathcal{Y}}$, which is what is expected in view of what we have discussed so far:
\begin{eqnarray}
 \widecheck{\mathcal{R}}_N &=&
 \frac{1}{\eta}\mathcal{H}_0\HH\D
+\frac{1}{2}\begin{bmatrix}
             \widetilde{k}_s^2&\\
                            & -\widetilde{k}_p^2
            \end{bmatrix}\HH\D_{-1}Q_N \eta
            + \begin{bmatrix}
              {\J}&\\
                            & {\J}
            \end{bmatrix}Q_N\eta,\quad\mathcal{H}_0= \begin{bmatrix}
 \mathrm{I}&-\mathrm{H}
 \label{eq:RcheckN}
 \\
 -\mathrm{H}&-\mathrm{I}
 \end{bmatrix} \\
  \widecheck{\mathcal{Y}}_N &=& \begin{bmatrix}
                                    \widecheck{\mathrm{Y}}_{p,N}&\\
                                   &\widecheck{\mathrm{Y}}_{s,N}
                                  \end{bmatrix},\quad \widecheck{\mathrm{Y}}_{k,N} := \eta^{-1} \D\HH +\frac{\widetilde{k}^2}2 \HH \D _{-1}Q_N\eta
 +\eta^{-1} \J Q_N
 \label{eq:YcheckN}
\end{eqnarray}

For these operators we can prove the following convergence result:
\begin{proposition}\label{prop:5.2}
For any $q+r>1/2$, $q\ge 0$
\begin{eqnarray}                                                                                                                        \label{eq:01:prop:5.2}
\| \widecheck{\V}_{k}^{(4)}-\widecheck{\V}_{k,N}^{(4)}\|_{H^{q+r}\to H^q}&\le&C_{q,r} N^{-r-\min\{q,4\}},\\                              \label{eq:02:prop:5.2}
\| \widecheck{\W}_{k}^{(2)}-\widecheck{\W}_{k,N}^{(2)}\|_{H^{q+r}\to H^q}&\le&C_{q,r} N^{-r-\min\{q,2\}},\\                              \label{eq:03:prop:5.2}
\| \widecheck{\K}^\top_{k}-\widecheck{\K}_{k,N}^\top \|_{H^{q+r}\to H^q}&\le&C_{q,r} N^{-r-\min\{q,3\}},\\                               \label{eq:04:prop:5.2}
\| \widecheck{\V}_{k,\ttt,\nnn}-\widecheck{\V}_{k,\ttt,\nnn,N} \|_{H^{q+r}\to H^q}&\le&C_{q,r} N^{-r-\min\{q,2\}},                      \label{eq:05:prop:5.2}
\end{eqnarray}
and
\begin{eqnarray}
\| \widecheck{\mathcal{Y}}_{N}-\widecheck{\mathcal{Y}} \|_{H^{q+r}\times H^{q+r}\to H^q\times H^{q}}&\le&C_{q,r} N^{-r-\min\{q,1\}},\\   \label{eq:06:prop:5.2}
\| \widecheck{\mathcal{R}}_{N}-\widecheck{\mathcal{R}} \|_{H^{q+r}\times  H^{q+r}\to H^q\times H^q}&\le&C_{q,r} N^{-r-\min\{q,1\}},      \label{eq:07:prop:5.2}
\end{eqnarray}
where $C_{q,r}$, possibly different in each occurrence, depends only on $q$, $r$.
\end{proposition}
\begin{proof}
These results follow from a careful application of the error estimates for the trigonometric interpolating operator cf. \eqref{eq:QN} and  Proposition \ref{prop:main:2a} to the  different terms involved which includes the new estimate
\begin{equation}\label{eq:Cetans:0}
 \| \mathrm{C}_{\eta}^{(2)}-\mathrm{C}_{\eta,N}^{(2)}\|_{H^{q+r}\to H^q}\le C N^{-r-\min\{q,2\}}.
\end{equation}
(Recall also Remark \ref{remark:convergence:estimates}).
\end{proof}

 Next, simple, but tedious calculations, show that  the principal part of \eqref{eq:firstEq:v2} can be rewritten as
\[
\begin{aligned}
{\widecheck{\mathcal{A}}_{\rm DL,pp}^{}}-  {\widecheck{\mathcal{A}}_{\rm SL,pp}^{}}\eta  \widecheck{\mathcal{Y}}
 \ =\ &
 \widecheck{\mathcal{A}}^{}_{\rm pp}-\frac1{ \eta} \mathcal{H}_0
\mathrm{J}\\
 & + \frac{1}{4}\begin{bmatrix}
                          & k_s^2\\
                    k_p^2 &
                   \end{bmatrix}
                   \left(  \eta^{-1} \HH\D_{-2}\eta^2\HH\D +\D_{-1}\eta+\eta^{-1} \HH\D_{-2}\eta^2\J\right)
                 \\
   & + \frac{1}4
                 \begin{bmatrix}
                     & \widetilde{k}_s^2\\
                     \widetilde{k}_p^2   &
                 \end{bmatrix}(\HH-\eta^{-1} \HH \eta) \HH\D_{-1}\eta\\
& + \frac{1}8\begin{bmatrix}
                         & k_s^2 \widetilde{k}_s^2\\
                   k_p^2 \widetilde{k}_p^2&
                  \end{bmatrix}\eta^{-1}\HH\D_{-2}\eta^3 \HH\D_{-1}\eta
 +\frac1{2\eta}\begin{bmatrix*}
                                     &  i\\
                                     i&
                                  \end{bmatrix*}
\mathrm{J}
+\frac1{2\eta} \mathcal{H}_0
\mathrm{J}
\\
&
 + \frac{1}4
                 \begin{bmatrix}
                     & i\widetilde{k}_s^2\\
                     i\widetilde{k}_p^2   &
                 \end{bmatrix} \eta^{-1}\J\eta \HH\D_{-1}\eta
\end{aligned}
\]
with
\[
\begin{aligned}
 \widecheck{\mathcal{A}}^{}_{\rm pp}\ =\ &
\frac{1}{\eta}\mathcal{H}_0 (\HH\D+\mathrm{J})
 +\frac{1}4\begin{bmatrix}
              (k_p^2+\widetilde{k}_p^2)\mathrm{I} & (k_s^2-\widetilde{k}_s^2)\HH\\
              (k_p^2-\widetilde{k}_p^2)\HH        & -(k_s^2+\widetilde{k}_s^2)\mathrm{I}
                                          \end{bmatrix}\HH\D_{-1}\eta\in\mathrm{OPS}(1).
\end{aligned}
\]
We can simply further the expressions above to facilitate the analysis. We first note that
\[
 \mathcal{H}_0
\mathrm{J}
= \begin{bmatrix}
           1&-i\\
           -i&-1
          \end{bmatrix}\J.
\]
On the other hand, recall  that $\mathrm{C}^{(2)}_a\in\mathrm{OPS}(-2)$ as specified in~\eqref{eq:Ca}. We can then define accordingly
\[
 \mathrm{C}^{(3)}_a := \HH\D_{-2}-a^{-1} \HH\D_{-2} a
 \in\mathrm{OPS}(-3),\quad
 \mathrm{C}^{(\infty)}_a := \HH -a^{-1} \HH  a\in\mathrm{OPS}(-\infty).
\]
Indeed, and since
\[
 \HH\D_{-2}\varphi:= -i\sum_{n\ne 0}\frac{1}{n^2}\widehat{\varphi}(n)e_{n}\deleted[id=vD]{(t)} \quad
 \HH \varphi:=
i \bigg[ \widehat{\varphi}(0)+\sum_{n\ne 0}\mathop{\rm \rm sign}(n)\widehat{\varphi}(n)e_{n}\deleted[id=vD]{(t)}\bigg]
\]
it is straightforward to show that (see also \eqref{eq:5.10c})
\begin{eqnarray}
 \mathrm{C}^{(3)}_a \varphi(t) &=&  i\left[  \varphi_{r_a}(t;0) -\varphi_{r_a}(t;1)e_1(t)+  \sum_{n\ne 0,1} \frac{2n-1}{n^2(n-1)^2} \varphi_{r_a}(t;n) e_n(t)
  \right]\label{eq:def:C3}
  \\
  \mathrm{C}^{(\infty)}_a\varphi(t) &=& 2i   \varphi_{r_a}(t;0) =\frac{i}{\pi} \int_{0}^{2\pi} r_a(t,\tau)\varphi(\tau)\,{\rm d}\tau.\label{eq:def:Cinfty}
\end{eqnarray}
Besides, noticing  that $\HH^2 =-\I$ and using the identity
\[
 \D_{-1}\eta  =   -\HH\D_{-2}\HH\D\eta=-\HH\D_{-2}\HH\eta'-\HH\D_{-2}\HH\eta\D=
 \D_{-2}\eta' -\HH\D_{-2}\HH\eta\D,
\]
we can rewrite
\[
\begin{aligned}
  \eta^{-1} \HH\D_{-2}\eta^2\HH\D +\D_{-1}\eta\ & =\   (\eta^{-1} \HH\D_{-2}\eta  -\HH\D_{-2})\eta\HH\D +
  \HH\D_{-2}\eta(\HH-\eta^{-1}\HH\eta)\D +\D_{-2}\eta' \\
  &=\ -  \mathrm{C}^{(3)}_{\eta}\eta\HH\D +\HH\D_{-2}\eta\mathrm{C}^{(\infty)}_{\eta}\D+\D_{-2}\eta'.
\end{aligned}
\]

In conclusion, we can write
 \[{\widecheck{\mathcal{A}}_{\rm DL,pp}^{}}-  {\widecheck{\mathcal{A}}_{\rm SL,pp}^{}}\eta  \widecheck{\mathcal{Y}}
 \ =\
 \widecheck{\mathcal{A}}_{\rm pp}+
 \widecheck{\mathcal{A}}_{\rm pp}^{(2)}
\]
where
\[
\begin{aligned}
 \widecheck{\mathcal{A}}^{(2)}_{\rm pp}\ =\ &  -\frac{1}{4}\begin{bmatrix}
                         & k_s^2\\
                   k_p^2 &
                  \end{bmatrix}
                  \left(\mathrm{C}^{(3)}_\eta\eta \HH\D
                  - \HH\D_{-2} \eta \mathrm{C}^{(\infty)}_{\eta}\D-  \D_{-2}\eta'  -
                   \eta^{-1}\HH \D_{-2}\eta^2 \J  \right)
                 \\
    &+  \frac{1}4
                 \begin{bmatrix}
                   & \widetilde{k}_p^2      \\
                     \widetilde{k}_s^2 &
                 \end{bmatrix}\mathrm{C}^{(\infty)}_\eta\HH\D_{-1}\eta
                 + \frac{1}8\begin{bmatrix}
                         & k_s^2 \widetilde{k}_s^2\\
                   k_p^2 \widetilde{k}_p^2&
                  \end{bmatrix}\eta^{-1}\HH\D_{-2}\eta^3 \HH\D_{-1}\eta
               + \frac12\begin{bmatrix*}
                                   -1 & 2i\\
                                   2i& 1
                                  \end{bmatrix*}
                                     \eta^{-1}\mathrm{J}\\
    & + \frac{1}4
                 \begin{bmatrix}
                   & i\widetilde{k}_s^2      \\
                     i\widetilde{k}_p^2 &
                 \end{bmatrix}\eta^{-1} \J\eta\HH\D_{-1}\eta  \\
\in \ &\mathrm{OPS}(-2).
\end{aligned}
\]
After (right-)multiplying the principal part $ \widecheck{\mathcal{A}}_{\rm pp}$  by the regularizer
\begin{equation}
\widecheck{\mathcal{R}} =
 \frac{1}{\eta}\mathcal{H}_0\HH\D
+\frac{1}{2}\begin{bmatrix}
             \widetilde{k}_s^2&\\
                            & -\widetilde{k}_p^2
            \end{bmatrix}\HH\D_{-1}\eta
            +  \begin{bmatrix}
              {\J}&\\
                            & {\J}
            \end{bmatrix}\eta
            \end{equation}
and using that $\mathcal{H}_0^2=0$,   we obtain that with the following operator
 \begin{equation}\label{eq:Hps}
 \mathcal{H}_{p,s}:=-\frac{1}4\begin{bmatrix}
  \left(k_p^2+k_s^2+\widetilde{k}_p^2+ \widetilde{k}_s^2\right)\mathrm{I}&
    -\left(k_p^2+k_s^2-\widetilde{k}_p^2- \widetilde{k}_s^2\right)\mathrm{H} \\
    \left(k_p^2+k_s^2-\widetilde{k}_p^2- \widetilde{k}_s^2\right)\mathrm{H}&
   \left(k_p^2+k_s^2+\widetilde{k}_p^2+ \widetilde{k}_s^2\right)\mathrm{I}
\end{bmatrix},
 \end{equation}
 which will turn out to be the principal part of $\widecheck{{\cal A}}_{\rm pp}\widecheck{\cal R}$ (see also Theorem \ref{theo:01}), it holds
\[
\begin{aligned}
 \widecheck{\mathcal{A}}^{}_{\rm pp}\widecheck{\mathcal{R}} \
 = \ &\mathcal{H}_{p,s} + \underbrace{\frac18 \begin{bmatrix}
     \widetilde{k}_s^2(k_p^2+\widetilde{k}_p^2)\I & -\widetilde{k}_p^2(k_s^2-\widetilde{k}_s^2)\HH \\
          \widetilde{k}_s^2(k_p^2-\widetilde{k}_p^2)\I & \widetilde{k}_p^2(k_s^2+\widetilde{k}_s^2)\HH \\
              \end{bmatrix} (\HH\D_{-1}\eta)^2}_{\widecheck{\mathcal{A}}_{\rm pp,\mathcal{R}}^{(2)}}
  \\
  &+\frac{1}2 \begin{bmatrix}
                    &   \widetilde{k}_p^2 \\
                    -\widetilde{k}_s^2  &
                  \end{bmatrix} \mathrm{C}_\eta^\infty
  +\frac{1}{\eta}\begin{bmatrix*}
                   \I & \HH \\
                  -\HH& \I
                 \end{bmatrix*}\D\mathrm{C}_\eta^\infty \eta^{-1} \HH\D
                 \\
              &
              +\begin{bmatrix}
                   1 & -i\\
                   -i & -1
\end{bmatrix}  \eta^{-1}\mathrm{J}\eta \begin{bmatrix}
                   \HH\D &\D\\
                   \D & - \HH\D
\end{bmatrix} \\
  &
   +\frac{1}4\begin{bmatrix}
  \left(k_p^2+k_s^2+\widetilde{k}_p^2- \widetilde{k}_s^2\right) &
    -\left(k_p^2+k_s^2+\widetilde{k}_p^2- \widetilde{k}_s^2\right)i \\
    \left(k_p^2+k_s^2-\widetilde{k}_p^2+ \widetilde{k}_s^2\right)i&
   \left(k_p^2+k_s^2-\widetilde{k}_p^2+ \widetilde{k}_s^2\right)
\end{bmatrix} \mathrm{J}
  \\
  &
  +\frac12 \begin{bmatrix}
  \widetilde{k}_s^2  &      \widetilde{k}_p^2  i \\
   -\widetilde{k}_s^2i&
    \widetilde{k}_p^2
\end{bmatrix}
\eta^{-1}\J\eta
  +\begin{bmatrix*}
                   1 & -i\\
                   -i&-1
                  \end{bmatrix*}\eta^{-1}\J\eta
     \\
  &
  + \frac{1}4\begin{bmatrix}
  \left(k_p^2+ \widetilde{k}_p^2\right)\I &
     \left(k_s^2- \widetilde{k}_s^2\right)\HH \\
    \left(k_p^2- \widetilde{k}_p^2\right)\HH&
   -\left(k_s^2+ \widetilde{k}_s^2\right)\I
\end{bmatrix} \HH\D_{-1}\eta\mathrm{J}\eta\\
     =\ & \mathcal{H}_{p,s} +  \widecheck{\mathcal{A}}_{\rm pp,\mathcal{R}}^{(2)} +\widecheck{\mathcal{A}}_{\rm pp,\mathcal{R}}^{(\infty)}.
\end{aligned}
\]
In short, the associated pseudodifferential operator to \eqref{eq:Acomb:02} can be expressed as
\begin{equation}\label{eq:5.25}
\widecheck{\mathcal{A}}_{\rm comb}\widecheck{\mathcal{R}}=
 \mathcal{H}_{p,s}+\underbrace{   \widecheck{\mathcal{A}}_{\rm pp,\mathcal{R}}^{(2)} +\widecheck{\mathcal{A}}_{\rm pp,\mathcal{R}}^{(\infty)}+\widecheck{\mathcal{A}}^{(2)}_{\rm pp} \widecheck{\mathcal{R}}  +\widecheck{\mathcal{A}}^{(2)}_{\rm comb}\widecheck{\mathcal{R}}}_{=: \mathcal{A}_{{\rm comb},\mathcal{R}}^{(1)}}
\end{equation}
where the key feature in the latter decomposition is
\[
  \mathcal{A}^{(1)}_{{\rm comb},\mathcal{R}}  \in\mathrm{OPS}(-1).
\]
We point out that the $ \mathcal{H}_{p,s}$ is a  continuous Fourier multiplier operator with a continuous inverse from $H^s\times H^s$ into itself. On the other hand, the operators $\widecheck{\mathcal{A}}_{\rm pp,\mathcal{R}}^{(\infty)}$ and $\widecheck{\mathcal{A}}_{\rm comb}^{(2)}$ depends only on $k_p$, $k_s$ and  $\eta$. Finally, it is the operator $\widecheck{\mathcal{A}}^{(2)}$ that retains the major dependence on the parametrization $\widecheck{\bf x}$ of the curve $\Gamma$ via the regular part of the BIOs.

In order to present our full Nystr\"om discretization of the regularized combined field equations \eqref{eq:ARcheck} we  introduce
\begin{subequations}\label{eq:PartsOfARN}
\begin{eqnarray}
\widecheck{\mathcal{A}}_{{\rm pp},\mathcal{R},N}^{(2)} &\ :=\ &
 \frac18 \begin{bmatrix}
     \widetilde{k}_s^2(k_p^2+\widetilde{k}_p^2)\I & -\widetilde{k}_p^2(k_s^2-\widetilde{k}_s^2)\HH \\
          \widetilde{k}_s^2(k_p^2-\widetilde{k}_p^2)\I & \widetilde{k}_p^2(k_s^2+\widetilde{k}_s^2)\HH \\
              \end{bmatrix} (\HH\D_{-1}Q_N \eta)^2
              \nonumber\\
              %
              %
\\
\widecheck{\mathcal{A}}_{{\rm pp},\mathcal{R},N}^{(\infty)}&\ :=\ &
\deleted[id=vD]{+}\frac{1}2 \begin{bmatrix}
                    &   \widetilde{k}_p^2 \\
                    -\widetilde{k}_s^2  &
                  \end{bmatrix} \mathrm{C}_\eta^\infty
  +\frac{1}{\eta}\begin{bmatrix*}
                   \I & \HH \\
                  -\HH& \I
                 \end{bmatrix*}\D Q_N \mathrm{C}_{\eta,N}^\infty \eta^{-1} \HH\D
                 \\
              &&
              +\begin{bmatrix}
                   1 & -i\\
                   -i & -1
\end{bmatrix}  \eta^{-1}\mathrm{J} Q_N \eta \begin{bmatrix}
                   \HH\D &\D\\
                   \D & - \HH\D
\end{bmatrix}
  \nonumber   \\
  &&
   +\frac{1}4\begin{bmatrix}
  \left(k_p^2+k_s^2+\widetilde{k}_p^2- \widetilde{k}_s^2\right) &
    -\left(k_p^2+k_s^2+\widetilde{k}_p^2- \widetilde{k}_s^2\right)i \\
    \left(k_p^2+k_s^2-\widetilde{k}_p^2+ \widetilde{k}_s^2\right)i&
   \left(k_p^2+k_s^2-\widetilde{k}_p^2+ \widetilde{k}_s^2\right)
\end{bmatrix} \mathrm{J} Q_N
  \nonumber   \\
  &&
  +\frac12 \begin{bmatrix}
  \widetilde{k}_s^2  &      \widetilde{k}_p^2  i \\
   -\widetilde{k}_s^2i&
    \widetilde{k}_p^2
\end{bmatrix}
\eta^{-1}\J  Q_N  \eta
  +\begin{bmatrix*}
                   1 & -i\\
                   -i&-1
                  \end{bmatrix*}\eta^{-1}\J  Q_N \eta
  \nonumber   \\
  &&
  + \frac{1}4\begin{bmatrix}
  \left(k_p^2+ \widetilde{k}_p^2\right)\I &
     \left(k_s^2- \widetilde{k}_s^2\right)\HH \\
    \left(k_p^2- \widetilde{k}_p^2\right)\HH&
   -\left(k_s^2+ \widetilde{k}_s^2\right)\I
\end{bmatrix} \HH\D_{-1} Q_N \eta\mathrm{J} Q_N \eta.\\
   \widecheck{\mathcal{A}}^{(2)}_{{\rm pp},N} &\ :=\ &  -\frac{1}{4}\begin{bmatrix}
                         & k_s^2\\
                   k_p^2 &
                  \end{bmatrix}
                  \left(\mathrm{C}^{(3)}_{\eta,N}\eta \HH\D
                  - \HH\D_{-2} \eta \mathrm{C}^{(\infty)}_{\eta,N}\D-  \D_{-2}Q_N\eta'  -
                   \eta^{-1}\D_{-2}Q_N\eta^2 \J   \right)
               \nonumber  \\
    &&+  \frac{1}4
                 \begin{bmatrix}
                   & \widetilde{k}_p^2      \\
                     \widetilde{k}_s^2 &
                 \end{bmatrix}\mathrm{C}^{(\infty)}_{\eta,N}\HH\D_{-1}Q_N\eta
                 + \frac{1}8\begin{bmatrix}
                         & k_s^2 \widetilde{k}_s^2\\
                   k_p^2 \widetilde{k}_p^2&
                  \end{bmatrix}\eta^{-1}\HH\D_{-2}Q_N\eta^3 \HH\D_{-1}Q_N \eta
\nonumber
 \\
 &&              + \frac12\begin{bmatrix*}
                                   -1 & 2i\\
                                   2i& 1
                                  \end{bmatrix*}
                                     \eta^{-1}\mathrm{J}
        + \frac{1}4 \begin{bmatrix}
                                          & i\widetilde{k}_s^2      \\
                       i\widetilde{k}_p^2 &
                  \end{bmatrix}\eta^{-1} \J Q_N \eta\HH\D_{-1}Q_N\eta
\\
\widecheck{\mathcal{R}}_N &\ :=\ &
 \frac{1}{\eta}\mathcal{H}_0\HH\D
+\frac{1}{2}\begin{bmatrix}
             \widetilde{k}_s^2&\\
                            & -\widetilde{k}_p^2
            \end{bmatrix}\HH\D_{-1}Q_N \eta
            + \begin{bmatrix}
              {\J}&\\
                            & {\J}
            \end{bmatrix}Q_N\eta,\quad\mathcal{H}_0= \begin{bmatrix}
 \mathrm{I}&-\mathrm{H} \\
 -\mathrm{H}&-\mathrm{I}
 \end{bmatrix} \\
{\widecheck{\mathcal{A}}_{{\rm comb},N}^{(2)}} &\ :=\ &    \begin{bmatrix}
     \widecheck{\W}^{(2)}_{p,N}                      &k_s^2 \widecheck{\V}_{s,\ttt,\nnn,N}-\widecheck{\K}_{s,N}^\top \eta^{-1}\D\\
     k_p^2 \widecheck{\V}_{p,\ttt,\nnn,N}-\widecheck{\K}^\top_{p,N}\eta^{-1}\D & -\widecheck{\W}_{s,N}^{(2)}\\
    \end{bmatrix} \nonumber\\
    &&-
    \begin{bmatrix}
     \widecheck{\K}^\top _{p,N} & \eta^{-1}{\D}Q_N  \widecheck{\V}_{s,N}^{(4)}\\
     \eta^{-1}{\D} Q_N \V_{p,N}^{(4)} & - \widecheck{\K}^\top_{s,N}
    \end{bmatrix}   \mathcal{Y}_N,
    \\
  \widecheck{\mathcal{Y}}_ {N}  &\ :=\ & \begin{bmatrix}
                                    \widecheck{\mathrm{Y}}_{p,N}&\\
                                   &\widecheck{\mathrm{Y}}_{s,N}
                                  \end{bmatrix},\quad \widecheck{\mathrm{Y}}_{k,N} = \eta^{-1} \D\HH +\frac{\widetilde{k}^2}2 \HH \D _{-1}Q_N\eta
 +\eta^{-1} \J Q_N
\end{eqnarray}
\end{subequations}
In the expression above,    $\mathrm{C}_{\eta,N}^{(\infty)}$ and $\mathrm{C}_{\eta,N}^{(3)}$ are the discretizations of $\mathrm{C}_{\eta}^{(\infty)}$ and $\mathrm{C}_{\eta}^{(3)}$ constructed using Proposition \ref{prop:main:2a}, as  $\mathrm{C}_{\eta,N}^{(2)}$, introduced in \eqref{eq:5.10c}, is of  $\mathrm{C}_{\eta}^{(2)}$.  Hence, it is easy to check that it holds
\begin{equation}\label{eq:Cetans}
 \| \mathrm{C}_{\eta}^{(\alpha)}-\mathrm{C}_{\eta,N}^{(\alpha)}\|_{H^{q+r}\to H^q}\le C N^{-r-\min\{\alpha,q\}},\quad \alpha\in\{3,\infty\}.
\end{equation}

Then the method is as follows: solve
\begin{equation}\label{eq:theMethod:check}
 \mathcal{A}_{{\rm comb},\mathcal{R},N}  \begin{bmatrix}
                                                            \widecheck{\lambda}_{p,N}\\
                                                            \widecheck{\lambda}_{s,N}
                                                           \end{bmatrix}
=
\left( \mathcal{H}_{p,s} + \mathcal{Q}_N
\mathcal{A}_{{\rm comb},\mathcal{R},N}^{(1)} \right)\begin{bmatrix}
                                                            \widecheck{\lambda}_{p,N}\\
                                                            \widecheck{\lambda}_{s,N}
                                                           \end{bmatrix}
=\mathcal{Q}_N
\begin{bmatrix}
     \widecheck{f}_{\nnn}\\
     \widecheck{f}_{\ttt}
   \end{bmatrix}.
\end{equation}
where
\begin{equation}\label{eq:ARN}
 \mathcal{A}_{{\rm comb},\mathcal{R},N}^{(1)}=
 \widecheck{\mathcal{A}}_{\rm pp,\mathcal{R},N}^{(2)}  +
 \widecheck{\mathcal{A}}_{{\rm pp},\mathcal{R},N}^{(\infty)}+\widecheck{\mathcal{A}}^{(2)}_{{\rm pp}, N}\widecheck{\mathcal{R}}_N  +\widecheck{\mathcal{A}}^{(2)}_{{\rm comb},N}\widecheck{\mathcal{R}}_N\in\mathrm{OPS}(-1)
\end{equation}
and construct next
\begin{equation}\label{eq:ARN:02}
\begin{bmatrix*}
\widecheck{\varphi}_{p,N}\\
\widecheck{\varphi}_{s,N}
\end{bmatrix*}:= \widecheck{\cal R}_N\begin{bmatrix}
                                                            \widecheck{\lambda}_{p,N}\\
                                                            \widecheck{\lambda}_{s,N}
                                                           \end{bmatrix}.
  \end{equation}
\deleted[id=vD]{Notice that again,}
Notice again that since ${\cal H}_{p,q}$ is a Fourier multiplier operator,  and by construction of the operators involved, the pairs $(\widecheck{\lambda}_{p,N},\ \widecheck{\lambda}_{s,N}), \ (\widecheck{\varphi}_{p,N},\ \widecheck{\varphi}_{s,N})$ belong to $\mathbb{T}_N\times \mathbb{T}_N$. That is, the  unknowns in \eqref{eq:ARN} and the densities in \eqref{eq:ARN:02} are uniquely determined by the values of these functions at the grid points $ \{m h\}$, $h=2\pi/N$.

Finally, we are in the position to establish  the approximation properties, in operator norm of  $\mathcal{A}_{{\rm comb},\mathcal{R},N }^{(1)}$ cf. \eqref {eq:ARN} to $\mathcal{A}_{{\rm comb},\mathcal{R} }^{(1)}$ cf. \eqref{eq:5.25}.

\begin{theorem} For any $q,\ r\ge 0$ with $q+r>3/2$, and $N$ sufficiently large, there exists $C_{q,r}$ so that
\begin{eqnarray*}
 \left\|\mathcal{A}_{{\rm comb},\mathcal{R} }^{(1)} -\mathcal{Q}_N \mathcal{A}_{{\rm comb},\mathcal{R},N}^{(1)}
 \right\|_{H^{q+r}\times H^{q+r} \to H^q\times H^q }
 \le C N^{-r-\min\{ 1,q-2\}}.
\end{eqnarray*}
\end{theorem}
\begin{proof}
By \eqref{eq:QN:02}
 \[
  \|\bm{\psi}-\mathcal{Q}_N \bm{\varphi}\|_{q}\le \|\mathcal{Q}_N (\bm{\psi}-\bm{\varphi})\|_{q}+ \|\bm\psi-\mathcal{Q}_N \bm\varphi\|_{q}
  \le C'\left(\|\bm{\psi}-\bm{\varphi}\|_{q}+ N^{-1}\|\bm{\psi}-\bm{\varphi}\|_{q+1}+N^{-r}\| \bm{\varphi}\|_{q+r} \right).
 \]
Setting $\bm{\psi} = \mathcal{A}_{{\rm comb},\mathcal{R} }^{(1)}(\varphi_p,\varphi_s)^\top$ and
 $\bm{\varphi}=  \mathcal{A}_{{\rm comb},\mathcal{R},N }^{(1)} (\varphi_p,\varphi_s)^\top$, and taking into account that  $\mathcal{A}_{{\rm comb},\mathcal{R} }^{(1)}\in\mathrm{OPS}(-1)$, we can reduce the result to bound
 \[
   \left\|\mathcal{A}_{{\rm comb},\mathcal{R} }^{(1)} - \mathcal{A}_{{\rm comb},\mathcal{R},N}^{(1)}
 \right\|_{H^{q+r}\times H^{q+r} \to H^q\times H^q }
 +
    N^{-1}\left\|\mathcal{A}_{{\rm comb},\mathcal{R} }^{(1)} - \mathcal{A}_{{\rm comb},\mathcal{R},N}^{(1)}
 \right\|_{H^{q+r}\times H^{q+r} \to H^{q+1}\times H^{q+1} }.
 \]
 The result follows from the following  estimates:
 \begin{enumerate}[label = (\alph*)]
 \item  Estimate
 \[
  \|\widecheck{\mathcal{A}}_{{\rm comb}}^{(2)}\widecheck{\mathcal{R}}-\widecheck{\mathcal{A}}_{{\rm comb},N}^{(2)}\widecheck{\mathcal{R}}_{k,N}\|_{H^{q+r}\times H^{q+r}\to H^{q}\times H^q} \le  C N^{-r-\min\{q-1,1\}}
 \]
which can be derived from Proposition \ref{prop:5.2}, specifically  estimate \eqref{eq:07:prop:5.2},
 \[
   \|\widecheck{\mathcal{Y}}-\widecheck{\mathcal{Y}}_{N}\|_{H^{q+r}\times H^{q+r} \to H^{q}\times H^q},
 \]
 with estimates \eqref{eq:02:prop:5.2}-\eqref{eq:05:prop:5.2}, that imply
\[
  \|\widecheck{\mathcal{A}}_{{\rm comb}}^{(2)}-\widecheck{\mathcal{A}}_{{\rm comb},N}^{(2)}\|_{H^{q+r-1}\times H^{q+r-1}\to H^{q}\times H^q}\le C N^{-r-\min\{q-1,1\}}
\]
 combined with estimate \eqref{eq:07:prop:5.2}
\begin{eqnarray}\label{eq:RN-R}
 \|\widecheck{\mathcal{R}}-\widecheck{\mathcal{R}}_{N}\|_{H^{q+r}\times H^{q+r} \to H^{q}\times H^q}&\le& C N^{r-\min\{q,1\}}.
 \end{eqnarray}
 (Recall that $ \widecheck{\mathcal{R}} \in \mathrm{OPS}(-1)$.)

 \item Estimate
\begin{eqnarray*}
 \|\widecheck{\mathcal{A}}_{{\rm pp}}^{(2)}\widecheck{\mathcal{R}}-\widecheck{\mathcal{A}}_{{\rm pp},N}^{(2)}\widecheck{\mathcal{R}}_N\|_{H^{q+r}\times H^{q+r}\to H^{q}\times H^q}&\le& C N^{-r-\min\{q-1,1\}}
\end{eqnarray*}
which can be derived from
\begin{eqnarray*}
 \|\widecheck{\mathcal{A}}_{{\rm pp}}^{(2)}-\widecheck{\mathcal{A}}_{{\rm pp},N}^{(2)}\|_{H^{q+r-1}\times H^{q+r-1}\to H^{q}\times H^q} \le  C N^{-r-\min\{q-1,1\}}
\end{eqnarray*}
 using \eqref{eq:Cetans}
and \eqref{eq:RN-R}.
\item Estimate
\begin{eqnarray*}
 \|\widecheck{\mathcal{A}}_{{\rm pp},\mathcal{R}}^{(2)}-\widecheck{\mathcal{A}}_{{\rm pp},\mathcal{R}}^{(2)}\|_{H^{q+r }\times H^{q+r }\to H^{q}\times H^q}&\le& C N^{-r-\min\{q-2,2\}}.
\end{eqnarray*}
which follow from the easy-to-prove estimate
\[
\begin{aligned}
  \|(\HH\D_{-1})^2- & (\HH\D_{-1}Q_N \eta)^2\|_{H^{q+r }\to H^{q }} \\
  & \le\ C \left(\|\HH\D_{-1}-  \HH\D_{-1}Q_N \|_{H^{q+r+1}\to H^{q }}+
  \|\HH\D_{-1}-  \HH\D_{-1}Q_N \|_{H^{q+r}\to H^{q-1 }}\right)\\
  &\le  C' N^{-r-\min\{2,q\}}
  \end{aligned}
\]
\item Estimate
\begin{eqnarray*}
 \|\widecheck{\mathcal{A}}_{{\rm pp},\mathcal{R}}^{(\infty)}-\widecheck{\mathcal{A}}_{{\rm pp},\mathcal{R}}^{(\infty)}\|_{H^{q+r }\times H^{q+r }\to H^{q}\times H^q}&\le& C N^{-r- q+1}.
\end{eqnarray*}
 \end{enumerate}

\end{proof}

As consequence we can state  stability and convergence of our Nystr\"om method:
\begin{theorem}\label{theo:5.4}
For $N$ large enough  the equations of the numerical method \eqref{eq:theMethod:check}  admits a unique solution $(\widecheck{\lambda}_{p,N},\widecheck{\lambda}_{s,N})$ which satisfies, for any $q>2$,
\[
\|\widecheck{\lambda}_{p,N}\|_{q}+\|\widecheck{\lambda}_{s,N}\|_{q}\le C_q \big(
\|\widecheck{\lambda}_{p}\|_{q}+\|\widecheck{\lambda}_{s}\|_{q}\big)\le C'_{q,r}  N^{-r}\big(
 \|  \widecheck{f}_{\nnn}\|_{q+r+1}+\|\widecheck{f}_{\ttt}\|_{q+r+1}\big)
\]
for any $\widecheck{\lambda}_p,$ $\widecheck{\lambda}_s$, solution of \eqref{eq:ARcheck}with $C_q>0$ independent of $N$ and of the true solution.
Furthermore, we have the following error estimate
\[
\| \widecheck{\lambda}_{p}-\widecheck{\lambda}_{p,N}\|_{q}+\|\widecheck{\lambda}_{s}-\widecheck{\lambda}_{s,N}\|_{q}\le C_{q,r} N^{-r}\big(
 \|  \widecheck{\lambda}_{p}\|_{q+r}+\|\widecheck{\lambda}_{s}\|_{q+r}\big),
\]
and, if $(\widecheck{\varphi}_p,\widecheck{\varphi}_s)$ are the densities given  by \eqref{eq:ARcheck2}, and $(\widecheck{\varphi}_{p,N},\widecheck{\varphi}_{s,N})$ those given by \eqref{eq:ARN:02}, it holds
\[
\|   \widecheck{\varphi}_{p}-\widecheck{\varphi}_{p,N}\|_{q}+\|\widecheck{\varphi}_{s}-\widecheck{\varphi}_{s,N}\|_{q}\le  C_{q,r}  N^{-r}\big(
 \|  \widecheck{\lambda}_{p}\|_{q+r+1}+\|\widecheck{\lambda}_{s}\|_{q+r+1}\big).
 \]
\end{theorem}
\begin{proof}
 Follows along the same lines as Theorem \ref{theo:4.5}. Notice that for the last result, and since $\widecheck{\mathcal{R}}_{N}$ is not longer a Fourier multiplier operator, we can use error estimate \eqref{eq:RN-R}.
\end{proof}


\section{Numerical experiments}\label{num}

We will present some numerical experiments for illustrating the theoretical results. First, we describe the considered domains for the different problems
\begin{enumerate}
\item The ellipsoid $\Gamma_{\rm e}$ centered at $(0,0)$ and semiaxes $(r,2r)$ with $r\approx 0.6485$ with the parametrization
\[
 \widecheck{{\bf x}}_{\rm e}(t):=r(\cos t,2 \sin t).
\]

\item The kite shaped curve $\Gamma_{\rm k}$ parameterized with
\[
 \widecheck{{\bf x}}_{\rm k}(t):= r\left(\cos t+\cos 2t,2\sin t\right),\quad r\approx 0.6348
\]

\item The cavity domain $\Gamma_{\rm c}$ given by the parametrization
\[
\widecheck{{\bf x}}_{\rm c}(t)= r\left(   \tfrac45\cos2t+\tfrac25\cos t),
                             \tfrac{7}{12}\sin t + \tfrac{17}{48}\sin2t+\tfrac{3}8\sin 3t - \tfrac{1}{24}\sin 4t  \right),\quad r\approx 0.6799
\]
\end{enumerate}

   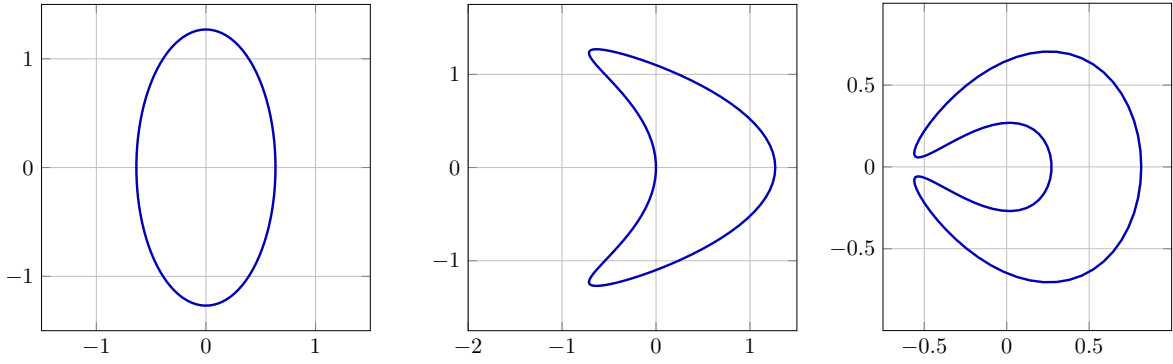
\begin{figure}
\begin{center}
\resizebox{0.30\textwidth}{!}{
\begin{tikzpicture}
  [declare function = {f1(\x)=1; }]
  \begin{axis}[
    xmin=-1.5 ,xmax=1.5 ,
    ymin=-1.5,ymax=1.5,
    grid=both,  axis equal image,
    xtick=    {-2, -1,0, 1,2 },
    ytick=    {-2, -1,0, 1,2 },
  ]
    \addplot[domain=0:360,samples=100,variable=\t,very thick, blue!70!black](%
      {1*0.6348*cos(t)}, %
      {2*0.6348*sin(t)});%
  \end{axis}
\end{tikzpicture}

}
\quad
 \resizebox{0.30\textwidth}{!}{
\begin{tikzpicture}
\begin{axis}[
  xmin=-2 ,xmax=1.5 ,
  ymin=-1.75,ymax=1.75,
  grid=both,  axis equal image,
xtick=    {-2, -1,0, 1,2 },
ytick=    {-2, -1,0, 1,2 },
]
  \addplot[domain=0:360,samples=100,variable=\t,very thick, blue!70!black](%
    {0.6348*(cos( t)+cos(2* t) )},
    {0.6348* 2* sin(t)}%
  );
\end{axis}
\end{tikzpicture}
}
 \resizebox{0.2825\textwidth}{!}{
\begin{tikzpicture}
\begin{axis}[
  xmin=-.75  ,xmax=1 ,
  ymin=-1  ,ymax=1,
  grid=both,  axis equal image,
xtick=    {  -.5,0,0.5 },
ytick=    {  -.5, 0,.5  },
]
  \addplot[domain=0:360,samples=100,variable=\t,very thick, blue!70!black](%
    {.6799 *((4*cos(2*t))/5+(2*cos(t))/5)}, %
    { .6799*((17*sin(2*t))/48+(3*sin(3*t))/8 %
                         -sin(4*t)/24+(7*sin(t))/12))}%
  );
\end{axis}

\end{tikzpicture}
}

\end{center}
\caption{\label{fig:geom} Geometries for the experiments considered in this section: The ellipse, the kite and the cavity domain. Notice that all the curves are of length $2\pi$.}
\end{figure}

In all these cases the curves are of length $2\pi$ (this constrained set the values of parameters $r$ above). The right-hand-side is taken so that
\[
 {\bf u}_{ s}({\bm x}) =  {\Phi}({\bm x}-{\bm x}_0)\begin{bmatrix}
                                                         1\\
                                                         1
                                                        \end{bmatrix}
\]
where $\Phi$ is the $2\times 2$ squared function matrix given
\[
\Phi({\bm x})=\frac{1}{\mu} \phi_0\left(k_s r\right) I_2+\frac{1}{\omega^2} \nabla_{\boldsymbol{x}} \nabla_{\boldsymbol{x}}^{\top}\left(\phi_0\left(k_s r\right)-\phi_0\left(k_p r\right)\right),
\]
the  fundamental solution for Navier equation. In all these cases ${\bm x}_0$ is a point taken in the interior of $\Gamma$. The Lam\'e parameters and the wave-numbers are taken to be
\[
 \lambda=2,\ \mu = 1.
\]
In the first series of experiments we have used arc-length parametrization for $N =
32, \ 64,$  $128,$ $\ldots,\ 1024$. We have measured the error far away from curve, specifically at 1024 points uniformly distributed at the circle of radius $4$ and centered at origin.
{The evaluation of the potentials is carried out by the rectangular rule which is optimal in the current frame: smooth periodic functions.
Let us emphasize that these points are sufficiently far away to benefit from the superconvergence of the algorithms involved, the numerical method and the quadrature rule. However, for points closer to the boundary we can expect, and it has been observed numerically, that this fast convergence is  achieved only for (very) large values of $n$. This phenomenon can be easily explained by the fact that the kernels (and their derivatives) of the Helmholtz single and double layer operators, although smooth, are nearly singular.   Naturally, this affects the convergence of the rectangular quadrature rule employed in evaluating the potentials. Certainly, this is not a new phenomenon in Boundary element methods, but something  quite common for which new careful strategies have to be considered such as increasing the number of quadrature points (evaluation of the density outside of the grid), special quadrature rules or techniques based on continuous expansion of the potentials to the boundary or interior of the domain. Although not used in our experiments, we are confident that these strategies could  also be applied in our cases.}

In Table \ref{tab:01}, we present the results for $\omega=10$ (which results in $k_p\approx 3.53$ and $k_s\approx 5.77$) and $\omega = 100$ (and so $k_p\approx 35.3$ and $k_s\approx 57.5$) for arc-length parametrizations. Fast, superalgebraically  convergence is observed for the ellipse, which is what theory predicts in view of Theorem \ref{theo:4.5}.
The convergence behavior for the kite and cavity curve varies significantly. We hypothesize that this poor convergence arises because the arc-length parametrization only introduces a large number of points at the complex parts of the domain for very large values of $N$, but the second and third derivatives of the arc-length, although formally smooth, are {\em very steep peak functions} at various points. This causes numerical instabilities that deteriorate the convergence of the method.

Same problems are solved in Table \ref{tab:02} but with the natural parametrizations{, i.e. the parameterizations used in the definition of the curves}. In these instances, the fast convergence stated in Theorem \ref{theo:5.4} is evident across all three cases examined. In this case the parametrizations $\widehat{\bf x}_{\rm k}
$, $\widehat{\bf x}_{\rm c}$ behave far better which makes the method converge at his full potential.

{We recall that the primary role of the regularizer is to render well-posed boundary integral formulations since
the principal symbol of these operators (in the pseudodifferential sense) is not an elliptic operator given that its kernel and coimage are not finite dimensional.
We emphasize that the primary function of the regularizer is to ensure well-posed boundary integral formulations. This is necessary because the principal symbol of these operators, in the pseudodifferential sense, is not  an elliptic operato since both the kernel and coimage of these operators are not finite-dimensional.

The preconditioner ${\cal R}$ remedies this defect by rendering the composition ${\cal A}_{\rm comb,\Gamma}{\cal R}$ operator elliptic. Furthermore, we showed how the regularizing operator ${\cal R}$ can lead to robust integral equations of the second-kind for the Helmholtz decomposition approach to the Navier equation, which are ideal for the analysis of Nystrom discretization (see Theorem \ref{theo:01} for arc-length parameterizations and \eqref{eq:Acomb:02} for arbitrary ones).

Our numerical experiments suggest that the use of regularized formulations is effective for low and mid-range frequencies $\omega$, that is the condition numbers of the regularized discrete formulations are lower than those of the combined field formulations, and the condition numbers of the regularized formulations do not depend on the mesh size.  We illustrate this behavior in Table \ref{tab:last:table} where the condition numbers of the BIE formulations are displayed for Experiment \#3 with the combined field formulation (no regularizer or ${\cal R}=I$) and the preconditioned/regularized formulations that use the operator ${\cal R}$ considered in this work. We list the condition numbers of the ensuing Nystrom for the two formulations, and for sufficiently large discretization size $N$, so that numerical error observed in the near field are below $10^{-8}$. We stress again that the continuous operator in the first case ${\cal R}=I$ is not Fredholm since it is a compact perturbation of a defective operator, i.e., an operator with an infinite dimensional kernel and coimage. The condition number of the Nystrom matrices corresponding to these operators grow with the discretization size $N$. Nevertheless, the combined formulation works sufficiently well in practice---that is their Nystrom matrices are invertible for fine enough discretization levels, but we have no mathematical explanation for this fact.

\begin{table}[h]
\centering
\begin{tabular}{|c|c|c|c|c|c|c|c|c|c|}
 \cline{3-8}
\multicolumn{2}{c|}{} & \multicolumn{6}{c|}{$n$} \\ \cline{3-8}
 \multicolumn{1}{c}{} &                                         &       32 &       64 &      128 &      256 &      512 &     1024     \\ \hline
\multirow{2}{*}{$\omega = 10$}&${\cal R}= {\cal I} $            & 2.79E+04 & 1.61E+05 & 9.51E+04 & 2.84E+05 & 1.37E+06& 6.21E+06  \\
                              &${\cal R}$                       & 1.28E+03 & 1.77E+03 & 2.96E+03 & 2.94E+03 & 2.94E+03& 2.94E+03  \\ \hline
\multirow{2}{*}{$\omega =40$} &${\cal R}= {\cal I}$             &    --    &    --    & 1.27E+04 & 2.13E+04 & 8.87E+04& 3.91E+05 \\
                              &${\cal R}$                       &    --    &    --    & 8.12E+03 & 8.12E+03 & 8.12E+03& 8.12E+03  \\ \hline
\multirow{2}{*}{$\omega =160$} &${\cal R}= {\cal I}$            &    --    &    --    &    --    & 7.95E+03 & 6.14E+04& 2.85E+04 \\
                              &${\cal R}$                       &    --    &    --    &    --    & 3.63E+05 & 3.63E+05& 3.63E+05  \\ \hline
\end{tabular}
\caption{Condition numbers of the regularized and the combined field formulation for Experiment 3}
\label{tab:last:table}
\end{table}

On the other hand, the regularized formulation leads to Nystrom matrices whose condition numbers are independent of discretization size $N$ and which, except in the high frequency domain, are order magnitude lower than those corresponding to the combined field formulations. However,  the condition number of the regularized formulations for $\omega = 160$ is significantly worse than for $\omega = 40$, but again we have no explanation for this fact. It is possible to construct different types of regularizers based on square root approximations of DtN operators---see \cite{DomTurc:2023}, and the use of such regularizers mitigate the conditioning of the regularized formulations in the high-frequency regime. The numerical analysis of such regularizers is a bit more complicated and we will carry it out in a different venue.
%
}

\begin{table}
  \begin{center}
\begin{tabular}{|c|c|c|c|c|c|c|c|c|c|c|}
\hline
$N$ & \multicolumn{2}{c|} {$\Gamma_{\rm e}$} &   \multicolumn{2}{c|} {$\Gamma_{\rm k}$} & \multicolumn{2}{c|} {$\Gamma_{\rm c}$}\\
\hline
 & $\omega=10$ & $\omega=100$ &  $\omega=10$ & $\omega=100$ &  $\omega=10$ & $\omega=100$   \\
\hline
32   &      1.42{\tt E-}05  &  4.88{\tt E-}02  &  1.08{\tt E-}02  &   6.69{\tt E-}02  &        4.44{\tt E-}03&   9.52{\tt E-}02\\
64   &      1.61{\tt E-}09  &  3.10{\tt E-}02  &  5.97{\tt E-}03  &   9.18{\tt E-}03  &        1.44{\tt E-}03&   2.98{\tt E-}02\\
128  &      3.63{\tt E-}14  &  3.29{\tt E-}04  &  2.84{\tt E-}04  &   2.31{\tt E-}03  &        8.28{\tt E-}04&   2.87{\tt E-}03\\
256  &      2.64{\tt E-}14  &  1.73{\tt E-}12  &  9.78{\tt E-}04  &   9.19{\tt E-}04  &        1.61{\tt E-}04&   1.56{\tt E-}04\\
512  &      4.20{\tt E-}14  &  1.11{\tt E-}12  &  1.50{\tt E-}06  &   1.11{\tt E-}04  &        4.93{\tt E-}06&   5.85{\tt E-}06\\
1024 &      5.11{\tt E-}14  &  6.94{\tt E-}13  &  3.42{\tt E-}07  &   2.12{\tt E-}06  &        4.62{\tt E-}09&   3.28{\tt E-}09\\
\hline
\end{tabular}
\caption{\label{tab:01}Error of the method for the three considered curves, $\Gamma_{\rm e}$ (ellipse), $\Gamma_{\rm k}$ (kite curve) and $\Gamma_{\rm c}$ the cavity problem for $\omega=10$ and $\omega=100$. In all these three cases the arc-length parametrization has been considered. }
\end{center}
\end{table}


\begin{table}
  \begin{center}
\begin{tabular}{|c|c|c|c|c|c|c|c|c|c|c|}
\hline
$N$ & \multicolumn{2}{c|} {$\Gamma_{\rm e}$} &   \multicolumn{2}{c|} {$\Gamma_{\rm k}$} & \multicolumn{2}{c|} {$\Gamma_{\rm c}$}\\
\hline
 & $\omega=10$ & $\omega=100$ &  $\omega=10$ & $\omega=100$ &  $\omega=10$ & $\omega=100$   \\
\hline
32   &    3.75{\tt E-}06  &     5.50{\tt E-}02  &   1.14{\tt E-}02  &   6.88{\tt E-}02  &  1.25{\tt E-}02  &   1.49{\tt E-}01\\
64   &    2.49{\tt E-}11  &     4.19{\tt E-}01  &   1.36{\tt E-}03  &   1.17{\tt E-}02  &  4.43{\tt E-}03  &   3.61{\tt E-}02\\
128  &    5.77{\tt E-}16  &     5.92{\tt E-}03  &   8.38{\tt E-}03  &   5.39{\tt E-}02  &  2.10{\tt E-}03  &   1.97{\tt E-}02\\
256  &    1.24{\tt E-}15  &     9.05{\tt E-}08  &   2.02{\tt E-}04  &   6.84{\tt E-}05  &  6.42{\tt E-}07  &   1.85{\tt E-}06\\
512  &    1.34{\tt E-}15  &     6.92{\tt E-}13  &   9.03{\tt E-}11  &   1.79{\tt E-}12  &  3.18{\tt E-}15  &   6.79{\tt E-}13\\
1024 &    2.10{\tt E-}15  &     3.89{\tt E-}13  &   6.79{\tt E-}15  &   5.97{\tt E-}13  &  1.06{\tt E-}15  &   4.67{\tt E-}13\\
\hline
\end{tabular}
\caption{\label{tab:02}Error of the method for the three considered curves, $\Gamma_{\rm e}$ (ellipse), $\Gamma_{\rm k}$ (kite curve) and $\Gamma_{\rm c}$ the cavity problem for $\omega=10$ and $\omega=100$. In all these three cases  the parametrization chosen is the {\em natural} one.}
\end{center}
\end{table}


We also point out that the underlying operator defined by the method is a compact perturbation of the invertible operator $
  \mathcal{H}_{p,s}$ (see Theorem \ref{theo:01} and \eqref{eq:Hps}). It is not difficult cf.~\cite{DomTurc:2023} to check that the eigenvalues of this operator are $\{ -(k_p^2 + k_s^2)/2,- (\widetilde{k}_p^2+ \widetilde{k}_s^2)/2 \}$   which implies that the eigenvalues of $\mathcal{A}_{{\rm comb},\mathcal{R} }$ are clustered around these points in the complex plane  Since  $\mathcal{A}_{{\rm comb},\mathcal{R},N }$ converges in operator norm to the continuous one,   the eigenvalues of the matrices of the numerical method inherit this property. As a consequence, we expect that in solution of the corresponding linear system Krylov iterative methods such as GMRES converge in a  low number of iterations which{,} moreover, is essentially independent of the level of discretizations.

This behaviour has been observed  for the three experiments as it is shown for  $N=1024$ (which means that 2048 eigenvalues are displayed for each matrix) and $\omega=10$ in  Figure \ref{fig:01} and $\omega =100$ in Figure \ref{fig:02}.

Regarding GMREs convergence, we displayed in Table \ref{tab:03} and \ref{tab:04} the iterations required to attain convergence, tolerance has been set in all the cases equal to $10^{-9}$, for arc-length and non-arc-lengh parametrizations. We observe that there are very minor gains in iteration counts when using the arclength parametrization for $\omega =10$, and no gains at all for the high-frequency problem.
{We emphasize that the   linear systems considered in this numerical section can still be solved straightforwardly using Gaussian elimination. Our goal in exploring the convergence behavior of GMRES is rather to investigate whether these techniques can potentially be applied in two-dimensional for very large problems and to to gain insight into what we can expect for three-dimensional problems. In the latter case, the use of iterative methods becomes almost mandatory, even for moderate values of $\omega$.}

In other words, the regularizer operator originally proposed for arc-length parametrizations has been successfully extended to arbitrary parametrizations.

{We emphasize that the   linear systems considered in this numerical section can still be solved straightforwardly using Gaussian elimination. Our goal in exploring the convergence behavior of GMRES is rather to investigate whether these techniques can potentially be applied in two-dimensional for very large problems and to to gain insight into what we can expect for three-dimensional problems. In the latter case, the use of iterative methods becomes almost mandatory, even for moderate values of $\omega$.}
\begin{figure}
 \[
  \includegraphics[width=0.5\textwidth]{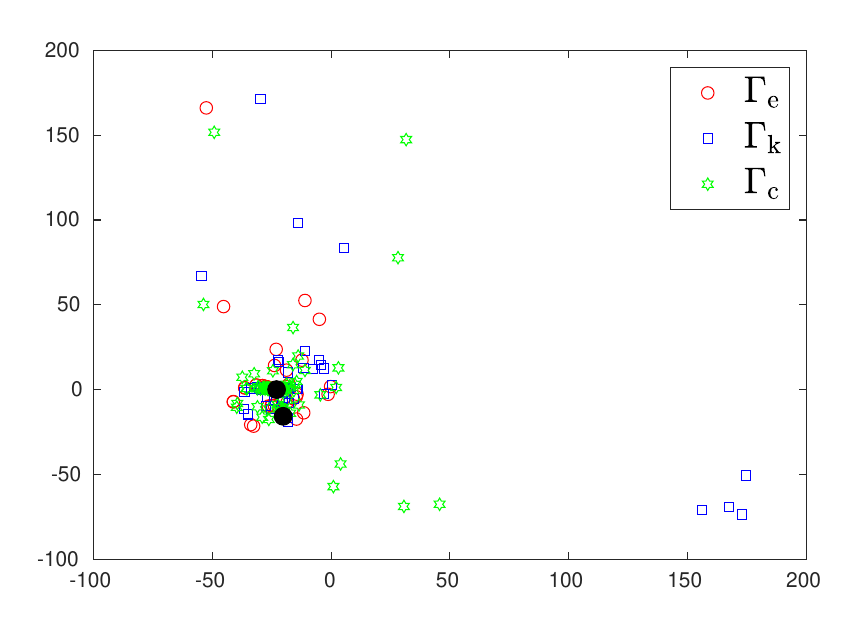}
 \quad
  \includegraphics[width=0.5\textwidth]{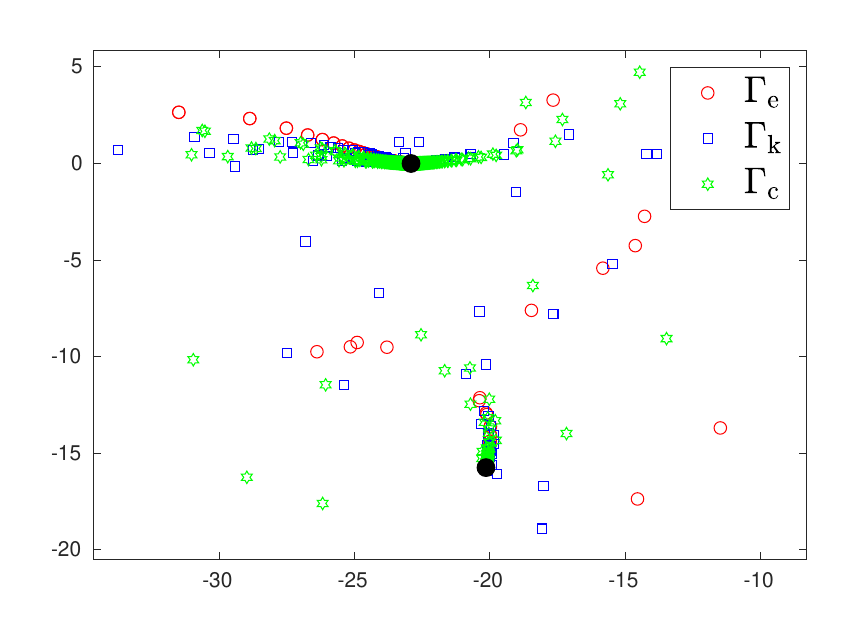}
 \]
\caption{\label{fig:01} Left panel: Eigenvalue distribution in the complex plane for the three geometries, with $N=1024$ and $\omega =10$, with {\em natural} parametrization. Right panel, a detail (zoom at) around the accumulation points, the eigenvalues of $\mathcal{H}_{p,s}$.}
\end{figure}
\begin{figure}
 \[
  \includegraphics[width=0.5\textwidth]{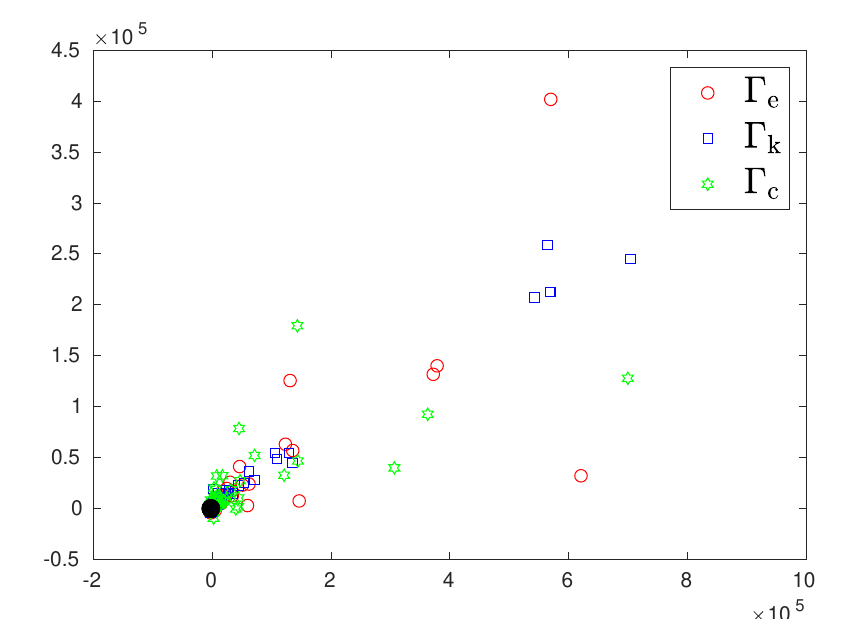}
 \quad
  \includegraphics[width=0.5\textwidth]{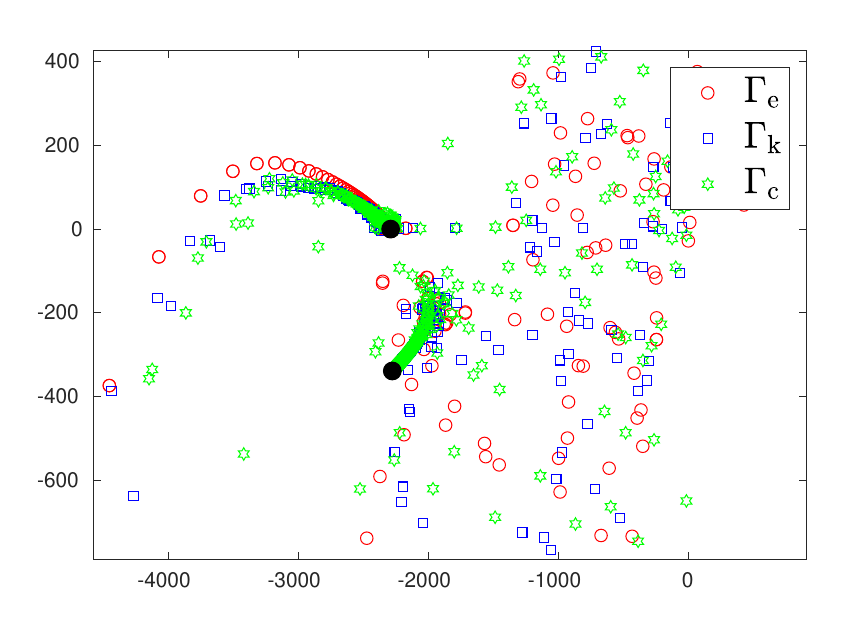}
 \]
\caption{\label{fig:02}Left panel: Eigenvalue distribution in the complex plane for the three geometries, with $N=1024$ and $\omega =100$. Right panel, a detail (zoom at) around the accumulation points, the eigenvalues of $\mathcal{H}_{p,s}$}
\end{figure}

\begin{table}
  \begin{center}
\begin{tabular}{|c|c|c|c|c|c|c|c|c|c|c|}
\hline
$N$ & \multicolumn{2}{c|} {$\Gamma_{\rm e}$} &   \multicolumn{2}{c|} {$\Gamma_{\rm k}$} & \multicolumn{2}{c|} {$\Gamma_{\rm c}$}\\
\hline
 & $\omega=10$ & $\omega=100$ &  $\omega=10$ & $\omega=100$ &  $\omega=10$ & $\omega=100$   \\
\hline
32   &          26    &      65   &    42   &   65     &  39   &           65 \\
64   &          24    &     129   &    42   &  129     &  39   &          129 \\
128  &          24    &     222   &    41   &  224     &  39   &          233 \\
256  &          24    &     234   &    40   &  237     &  39   &          249 \\
512  &          24    &     233   &    40   &  238     &  39   &          249 \\
1024 &          24    &     233   &    40   &  238     &  39   &          249 \\
\hline
\end{tabular}
\end{center}
\caption{\label{tab:03}Number of GMRES iterations required to achieve convergence, using a tolerance level of $10^{-9}$, for the considered geometries. $\Gamma_{\rm e}$ (ellipse), $\Gamma_{\rm k}$ (kite curve) and $\Gamma_{\rm c}$ (the cavity problem) for $\omega=10$ and $\omega=100$. In all these three cases the arc-length parametrization has been used.}
\end{table}

\begin{table}
\begin{center}
\begin{tabular}{|c|c|c|c|c|c|c|c|c|c|c|}
\hline
$N$ & \multicolumn{2}{c|} {$\Gamma_{\rm e}$} &   \multicolumn{2}{c|} {$\Gamma_{\rm k}$} & \multicolumn{2}{c|} {$\Gamma_{\rm c}$}\\
\hline
 & $\omega=10$ & $\omega=100$ &  $\omega=10$ & $\omega=100$ &  $\omega=10$ & $\omega=100$   \\
\hline
32   &             34    &      65   &  48    &   65    &   50   &    65\\
64   &             34    &     129   &  49    &  129    &   49   &   129\\
128  &             34    &     226   &  50    &  231    &   47   &   235\\
256  &             34    &     234   &  46    &  240    &   44   &   250\\
512  &             34    &     233   &  46    &  238    &   44   &   248\\
1024 &             34    &     233   &  46    &  235    &   44   &   248\\
\hline
\end{tabular}
\end{center}
\caption{\label{tab:04}Number of GMRES iterations required to achieve convergence, using a tolerance level of $10^{-9}$, for the considered geometries. $\Gamma_{\rm e}$ (ellipse), $\Gamma_{\rm k}$ (kite curve) and $\Gamma_{\rm c}$ (the cavity problem) for $\omega=10$ and $\omega=100$. In all these three cases the {\em natural parametrization} has been used.}
\end{table}

%
%
%
%
\section{Conclusions}
We analyzed in this paper Nystr\"om discretizations of regularized combined field integral formulations for the solution of Navier scattering problems with smooth boundaries and Dirichlet boundary conditions using the Helmholtz decomposition approach in two dimensions. In order to deliver integral equations of the second kind the regularization strategy we propose relies on compositions of the classical Helmholtz BIOs with approximations of DtN operators. We present and analyze in this paper stable discretizations of these compositions of pseudodifferential operators of opposite orders, both in the simpler case of arclength boundary parametrizations as well as in the more challenging case of general smooth parametrizations. The main idea in the analysis is to isolate via logarithmic kernel splittings the principal parts of the pseudodifferential operators involved and compute explicitly their compositions in the framework of Fourier multipliers. The operator composition of more regular remainders, which are all pseudodifferential operators of negative orders, amounts to simple Nystr\"om matrix multiplication and is amenable to a rather straightforward stability analysis. Extensions to the case of Neumann boundary conditions is currently underway.

\paragraph{Acknowledgments} The first author thanks the support of  projects ``Adquisición de conocimiento y minería de datos, funciones especiales y métodos numéricos avanzados'' from Universidad P\'{u}blica de Navarra, Spain and ``Técnicas innovadoras para la resolución de problemas
evolutivos'', ref. PID2022-136441NB-I00 from Ministerio de Ciencia e Innovación, Gobierno de España, Spain.
Catalin Turc gratefully acknowledges support from NSF through contract ``Optimized Domain Decomposition Methods for Wave Propagation in Complex Media'' ref DMS-1908602.

{The authors would like to thank the reviewers for their careful review and valuable suggestions, which helped us to improve the quality of this manuscript.}


\begin{thebibliography}{10}

\bibitem{AmKaLe:2009}
H.~Ammari, H.~Kang, and H.~Lee.
\newblock {\em Layer potential techniques in spectral analysis}, volume 153 of
  {\em Mathematical Surveys and Monographs}.
\newblock American Mathematical Society, Providence, RI, 2009.

\bibitem{MR0717691}
{D}.~N. Arnold and {W}.L. Wendland.
\newblock On the asymptotic convergence of collocation methods.
\newblock {\em Math. Comp.}, 41(164):349--381, 1983.

\bibitem{boubendir2015regularized}
Y.~Boubendir, V.~Dom\'{\i}nguez, D.~Levadoux, and C.~Turc.
\newblock Regularized combined field integral equations for acoustic
  transmission problems.
\newblock {\em SIAM Journal on Applied Mathematics}, 75(3):929--952, 2015.

\bibitem{MR3343368}
{Y}. Boubendir, {V}. Dom\'{\i}nguez, David Levadoux, and Catalin Turc.
\newblock Regularized combined field integral equations for acoustic
  transmission problems.
\newblock {\em SIAM J. Appl. Math.}, 75(3):929--952, 2015.

\bibitem{MR3463450}
Y.~Boubendir, C.~Turc, and V.~Dom\'{\i}nguez.
\newblock High-order {N}ystr\"{o}m discretizations for the solution of integral
  equation formulations of two-dimensional {H}elmholtz transmission problems.
\newblock {\em IMA J. Numer. Anal.}, 36(1):463--492, 2016.

\bibitem{BoTuDo:2015}
Y.~Boubendir, C.~Turc, and V.~Domínguez.
\newblock {High-order Nyström discretizations for the solution of integral
  equation formulations of two-dimensional {H}elmholtz transmission problems}.
\newblock {\em IMA Journal of Numerical Analysis}, 36(1):463--492, 03 2015.

\bibitem{BrackhageWerner}
H.~Brakhage and P.~Werner.
\newblock \"{U}ber das {D}irichletsche {A}ussenraumproblem f\"ur die
  {H}elmholtzsche {S}chwingungsgleichung.
\newblock {\em Arch. Math.}, 16:325--329, 1965.

\bibitem{osti_10185828}
{O}.~P. Bruno and {T.} Yin.
\newblock Regularized integral equation methods for elastic scattering problems
  in three dimensions.
\newblock {\em Journal of Computational Physics}, 410:109350, 2020.

\bibitem{chaillat2015approximate}
S.~Chaillat, M.~Darbas, and F.~Le~Lou{\"e}r.
\newblock Approximate local {D}irichlet-to-{N}eumann map for three-dimensional
  time-harmonic elastic waves.
\newblock {\em Computer Methods in Applied Mechanics and Engineering},
  297:62--83, 2015.

\bibitem{chaillat2020analytical}
S.~Chaillat, M.~Darbas, and F.~Le~Lou{\"e}r.
\newblock Analytical preconditioners for {N}eumann elastodynamic boundary
  element methods.
\newblock {\em Partial Differ. Equ. Appl.}, 2(2):Paper No. 22, 26, 2021.

\bibitem{chapko2000numerical}
{R.} Chapko, {R}. Kress, and {L}. Monch.
\newblock On the numerical solution of a hypersingular integral equation for
  elastic scattering from a planar crack.
\newblock {\em IMA journal of numerical analysis}, 20(4):601--619, 2000.

\bibitem{doi:10.1137/0519043}
M.~Costabel.
\newblock Boundary integral operators on lipschitz domains: Elementary results.
\newblock {\em SIAM Journal on Mathematical Analysis}, 19(3):613--626, 1988.

\bibitem{MR3141718}
{V}. Dom\'{\i}nguez, S.L. Lu, and {F.-J.} Sayas.
\newblock A {N}ystr\"{o}m flavored {C}alder\'{o}n calculus of order three for
  two dimensional waves, time-harmonic and transient.
\newblock {\em Comput. Math. Appl.}, 67(1):217--236, 2014.

\bibitem{dominguez2012fully}
{V.} Dom\'{\i}nguez, {S.}L. Lu, and {F.J.} Sayas.
\newblock A fully discrete {C}alder{\'o}n calculus for two dimensional time
  harmonic waves.
\newblock {\em Int. J. Numer. Anal. Model.}, 11(2):332--345, 2014.

\bibitem{dominguez2014nystrom}
{V.} Dom{\'\i}nguez, S.L. Lu, and {F.J.} Sayas.
\newblock A {N}ystr{\"o}m flavored {C}alder{\'o}n calculus of order three for
  two dimensional waves, time-harmonic and transient.
\newblock {\em Computers \& Mathematics with Applications}, 67(1):217--236,
  2014.

\bibitem{dominguez2008dirac}
{V.} Dom{\'\i}nguez, {M.L.} Rap{\'u}n, and {F.J.} Sayas.
\newblock Dirac delta methods for {Helm}holtz transmission problems.
\newblock {\em Advances in Computational Mathematics}, 28(2):119--139, 2008.

\bibitem{DoSaSa:2015}
V.~Dom\'{\i}nguez, T.~S\'{a}nchez-Vizuet, and {F.J.} Sayas.
\newblock A fully discrete {C}alder\'{o}n calculus for the two-dimensional
  elastic wave equation.
\newblock {\em Comput. Math. Appl.}, 69(7):620--635, 2015.

\bibitem{DomCat:2015}
V.~Dom\'{\i}nguez and C.~Turc.
\newblock High order {N}ystr\"{o}m methods for transmission problems for
  {H}elmholtz equations.
\newblock In {\em Trends in differential equations and applications}, volume~8
  of {\em SEMA SIMAI Springer Ser.}, pages 261--285. Springer, [Cham], 2016.

\bibitem{dominguez2021boundary}
V.~Dom\'{\i}nguez and C.~Turc.
\newblock Boundary integral equation methods for the solution of scattering and
  transmission 2{D} elastodynamic problems.
\newblock {\em IMA J. Appl. Math.}, 87(4):647--706, 2022.

 

\bibitem{DomTurc:2023}
V.~Dom\'{\i}nguez and C.~Turc.
\newblock Robust boundary integral equations for the solution of elastic
  scattering problems via {H}elmholtz decompositions.
\newblock {\em Comput. Math. Appl.}, 2024.
\newblock Accepted. Preprint available on arXiv:2211.16168.



\bibitem{dong2021highly}
H.~Dong, {J.} Lai, and {P.} Li.
\newblock A highly accurate boundary integral method for the elastic obstacle
  scattering problem.
\newblock {\em Mathematics of Computation}, 90:2785--2814, 2021.

\bibitem{faria2021general}
L.~M. Faria, C.~P{\'e}rez-Arancibia, and M.~Bonnet.
\newblock General-purpose kernel regularization of boundary integral equations
  via density interpolation.
\newblock {\em Computer Methods in Applied Mechanics and Engineering},
  378:113703, 2021.

\bibitem{hsiao2008boundary}
G.C. Hsiao and W.L. Wendland.
\newblock {\em Boundary integral equations}.
\newblock Springer, 2008.

\bibitem{KressH}
R.~Kress.
\newblock On the numerical solution of a hypersingular integral equation in
  scattering theory.
\newblock {\em J. Comput. Appl. Math.}, 61(3):345--360, 1995.

\bibitem{Kress}
R.~Kress.
\newblock {\em Linear Integral Equations}.
\newblock Applied Mathematical Sciences. Springer New York, 2013.

\bibitem{kupradze2012three}
V.D. Kupradze.
\newblock {\em Three-dimensional problems of elasticity and thermoelasticity}.
\newblock Elsevier, 2012.

\bibitem{Kussmaul}
R.~Kussmaul.
\newblock Ein numerisches {V}erfahren zur {L}\"{o}sung des {N}eumannschen
  {A}u\ss enraumproblems f\"{u}r die zweidimensionale {H}elmholtzsche
  {S}chwingungsgleichung.
\newblock In {\em Methoden und {V}erfahren der {M}athematischen {P}hysik,
  {B}and 1 ({B}ericht \"{u}ber eine {T}agung, {O}berwolfach, 1969)}, volume
  720/720a* of {\em B. I. Hochschulskripten}, pages 15--31. Bibliographisches
  Inst., Mannheim-Vienna-Z\"{u}rich, 1969.

\bibitem{MR0147661}
E.~Martensen.
\newblock \"{U}ber eine {M}ethode zum r\"{a}umlichen {N}eumannschen {P}roblem
  mit einer {A}nwendung f\"{u}r torusartige {B}erandungen.
\newblock {\em Acta Math.}, 109:75--135, 1963.

\bibitem{mclean:2000}
W.~McLean.
\newblock {\em Strongly elliptic systems and boundary integral equations}.
\newblock Cambridge University Press, Cambridge, 2000.

\bibitem{MR0958484}
K.~Ruotsalainen and J.~Saranen.
\newblock A dual method to the collocation method.
\newblock {\em Math. Methods Appl. Sci.}, 10(4):439--445, 1988.

\bibitem{Saranen}
J.~Saranen and G.~Vainikko.
\newblock {\em Periodic integral and pseudodifferential equations with
  numerical approximation}.
\newblock Springer Monographs in Mathematics. Springer-Verlag, Berlin, 2002.

\bibitem{deltaBEM}
F.-J. Sayas et~al.
\newblock {deltaBEM}: a {MATLAB}-based suite for 2-{D} numerical computing with
  the boundary element method on smooth geometries and open arcs.
\newblock In \url{https://team-pancho.github.io/deltaBEM/}. Accessed on Date.

\bibitem{MR1626332}
{I}~H. Sloan and W.~L. Wendland.
\newblock Qualocation methods for elliptic boundary integral equations.
\newblock {\em Numer. Math.}, 79(3):451--483, 1998.

\end{thebibliography}

\end{document}